\documentclass[11pt,oneside,reqno]{amsart}
\usepackage{amssymb,epsfig,amsmath,amsthm,amsopn,amstext,amsfonts,amsbsy,graphicx}
\usepackage{color}
\usepackage{bbm}
\usepackage[latin1]{inputenc}
\usepackage[T1]{fontenc}
\usepackage[english]{babel}
\input epsf

\def\Pi{{\mathtt P}_\infty}

\makeatletter
\def\moverlay{\mathpalette\mov@rlay}
\def\mov@rlay#1#2{\leavevmode\vtop{%
   \baselineskip\z@skip \lineskiplimit-\maxdimen
   \ialign{\hfil$\m@th#1##$\hfil\cr#2\crcr}}}
\newcommand{\charfusion}[3][\mathord]{
    #1{\ifx#1\mathop\vphantom{#2}\fi
        \mathpalette\mov@rlay{#2\cr#3}
      }
    \ifx#1\mathop\expandafter\displaylimits\fi}
\makeatother
\newcommand{\bigcupdot}{\charfusion[\mathop]{\bigcup}{\cdot}}
\newcommand{\dotcup}{\ensuremath{\mathaccent\cdot\cup}}

\newtheorem{theorem}{Theorem}
\newtheorem{proposition}{Proposition}
\newtheorem{lemma}{Lemma}

\newtheorem{corollary}{Corollary}

\title[Stable limit laws   for reaction-diffusion in random environment]{
Stable limit laws and structure of the scaling function for reaction-diffusion in random environment}

\author{G\'erard Ben Arous$^1$, Stanislav Molchanov and Alejandro F. Ram\'\i rez$^2$}

\thanks{ AMS 2010 {\it subject classifications}. Primary  82B41, 82B44 ;
secondary  60J80, 82C22.}

\thanks{$^1$Partially supported by NSF DMS1209165 and BSF 2014019
}

\thanks{$^2$Partially supported by Fondo Nacional de Desarrollo Cient\'\i fico
y Tecnol\'ogico 1020686 and Iniciativa Cient\'\i fica Milenio NC120062}


\address[G\'erard Ben Arous]{Courant Institute of Mathematical
Sciences\\
New York University\\
NY, NY 10012, USA}

\address[Stanislav Molchanov]{Department of Mathematics\\
University of North Carolina-Charlotte\\
376 Fretwell Bldg. 9201 University City Blvd.\\
Charlotte, NC 28223-0001, USA\\
and National Research University Higher School of Economics\\
Myasnitskaya ul., 20, Moscow, Russia, 101000}

\address[Alejandro Ram\'\i rez]{Facultad de Matem\'aticas\\
Pontificia Universidad Cat\'olica de Chile\\
Santiago 6904411, Chile\newline
Url: {\rm http://www.mat.puc.cl/\~\ \!\!\!\! aramirez}}

\bigskip
\email{benarous@cims.nyu.edu, smolchan@math.uncc.edu,aramirez@mat.puc.cl}

\bigskip

\begin{document}
\begin{abstract} 
We prove  the emergence of stable fluctuations for reaction-diffusion in random environment with Weibull tails.
This completes our work around the quenched to annealed transition phenomenon in this context of reaction diffusion. In \cite{bmr2}, we had already considered 
the model treated here and had studied fully the regimes where the law of large numbers is satisfied and where the fluctuations are Gaussian, but we had left open the regime of stable fluctuations. Our work is based on a spectral approach centered on the classical theory of rank-one perturbations. It illustrates the gradual emergence of the role of the higher peaks of the environments. This approach also allows us to give the delicate exact asymptotics of the normalizing constants needed in the stable limit law.

\end{abstract}
  
\maketitle

\section{Introduction}

We establish the emergence of stable fluctuations for branching random walks in a random environment.
More precisely, we consider branching random walks on the lattice $\mathbb Z^d$, and denote by $v(x)$ the rate of branching at site $x \in \mathbb Z^d$. We assume that the branching is binary (each particle gives birth to two offsprings) and that the rates  $(v(x))_{x \in \mathbb Z^d}$ are i.i.d random variables, which we call here the random environment.

We are interested in the spatial fluctuations of the number $N_x(t)$ of particles at time $t>0$, whose ancestor at time $0$ was at site $x \in \mathbb Z^d$, or rather in the behavior of its mean $m(t,x)$, as a function of the "quenched" random environment.

This question can be formulated as an equivalent problem about reaction-diffusion in random environment, since the random function $m(x,t)$ is the solution of the reaction-diffusion equation
\begin{equation}
\label{i1}
\frac{\partial m(x,t)}{\partial t}=\kappa\Delta m(x,t)+v(x)m(x,t),\qquad t\ge 0, x\in{\mathbb Z}^d,
\end{equation}
with initial condition,

\begin{equation}
\label{i1.111}
m(x,0)=1,\qquad x\in{\mathbb Z}^d,
\end{equation}
where  $\Delta$ is the discrete Laplacian on ${\mathbb Z}^d$, defined for
every function $f\in l^1({\mathbb Z}^d)$ as

$$
\Delta f(x):=\sum_{e\in E}(f(x+e)-f(x)),
$$
where $E:=\{e\in\mathbb Z^d: |e|_1=1\}$ and $|\cdot |_1$ is
the $l^1(\mathbb Z^d)$ norm.

Following \cite{bmr2}, we will study the mean number of particles at time t, if one starts with a particle at time 0, whose position is picked uniformly at random in the box $\Lambda_L$ of size $L$. More precisely we will consider the spatial average
\begin{equation}
m_L(t)=\frac{1}{|\Lambda_L|}\sum_{x\in\Lambda_L}m(x,t,v)
\end{equation}

This quantity exhibits a very rich dynamical transition, when $t$ increases, as a function of $L$. This transition was introduced a decade ago, as a mechanism for the "transition from annealed to quenched asymptotics" for Markovian dynamics, in \cite{bmr} in the context of random walks on random obstacles, and in \cite{bmr2} in the current context of reaction-diffusion in random environment. The basic intuition behind this rich picture can in fact be understood in the much simpler context of sums of i.i.d random exponentials, as was done in  \cite{bbm05}. In this simple context of i.i.d random variables, the transition boils down to the graduate emergence of the role of extreme values, which gradually impose a breakdown of the CLT first and eventually of the LLN, and induce stable fluctuations.

In the present context of branching random walks in random environment, we have proved in \cite{bmr2}, that the extreme values of the random environment play a similar role, and established the precise breakdown of the Central Limit Theorem and of the Law of Large Numbers. But we had left open the much more delicate question of stable fluctuations. This is the purpose of this work.

Our main result in this work establishes that, when the tails of the branching rate are Weibull distributed (as in \cite{bmr2}) and in the regime where the Central Limit Theorem fails, $m_L(t)$, once properly centered and scaled, converges to a stable distribution. Before establishing these stable limit laws, one major difficulty is to understand the needed exact asymptotics of the centering and scaling constants, as functions of $t$. This asymptotic behavior is unusually delicate, as we will see below.
The validity of these stable limits was until now only proved in two much simpler situations: either the simple case treated in \cite{bbm05} of sums of i.i.d random variables, or in the one dimensional case of random walks among random obstacles treated in \cite{bmr}. 
In a future work, we plan to extend our results to branching rates
with double exponentially decaying tails.

Our result was also claimed in the recent work of Gartner and Schnitzler in 
\cite{gs15}. Nevertheless we believe that the proof given in \cite{gs15} is incomplete.
In fact the statement of  Theorem 1 in \cite{gs15}, 
giving the stable limit law,  is ambiguous. The problem lays with the understanding of the normalizing function. In fact, this is one of the crucial places
 where in this article we use decisively the rank-one perturbation theory. 
 We believe that the estimate of the normalizing function
 (called here $e^{th(t)}$ and there $B_\alpha(t)$) is flawed.

The main tool here is naturally based on spectral theory, and more precisely on the classical theory of rank-one perturbations for self-adjoint operators, which we recall briefly in the Appendix. Using rank-one perturbation theory in this context of random media is quite natural, as for instance in \cite{bk16} (Biskup-Konig), where the case of faster tails (doubly exponential) is studied in great depth for the parabolic Anderson model. More broadly, spectral tools have been pushed quite far for the understanding of the extreme balues of the spectrum of the Anderson operator in the beautiful series of work \cite{a07,a08,a12,a13,a16} (Astrauskas).

Let us now describe more precisely the content of this article. We state precisely our results in Section2. First, in Section 2.1, we begin by giving our notations, and then state, in Section 2.2  our results for the exact behavior of the centering constant. In Section 2.3 we give our results about the scaling function, and finally in Section 2.4 we give our main result, i.e the convergence to a stable distribution for $m_L(t)$, once properly centered and scaled.
In Section 3, we establish the needed spectral results about our random Schrodinger operator, using the rank-one perturbation theory recalled in Appendix A. In Section 4, we recall the basic facts of extreme value theory for i.i.d Weibull distributions. And finally in Section 5, we establish our Theorem 1 about the behavior of the centering. In Section 6, we establish our Theorem 2 re the behavior of the scaling function. And finally in Section 7 we prove the main result, Theorem 3, establishing stable limit laws.

\medskip

\section{Notations and results}
\label{secnr}
 Let
  $v:=\{v(x):
x\in{\mathbb Z}^d\}$ where
$v(x)\in [0,\infty)$. Consider the space $W:=[0,\infty)^{{\mathbb Z}^d}$,
endowed with its natural $\sigma$-algebra. Let  $\mu$ be the probability
measure on $W$ such that the coordinates of $v$ are i.i.d.
Let us now consider a simple symmetric random walk of rate $\kappa>0$,
which  branches at a site $x$ at rate $v(x)$ giving birth to two particles. Let us call
$m(x,t,v)$ the expectation of the total
 number of random walks at time $t$ given that initially there was only one
random walk at site $x$, in the environment $v$. We will
frequently write $m(x,t)$ instead of $m(x,t,v)$.
In Propositions 1 and 8 of \cite{bmr2} it is shown that $\mu$-a.s. this expectation is finite
and   that  it
 satisfies the  parabolic Anderson
equations (\ref{i1}) and (\ref{i1.111}).

 Throughout, we will call $v$ the {\it potential}
of the equation (\ref{i1}) and we will assume that
it has a Weibull law of parameter $\rho>1$ so that

\begin{equation}
\label{weibul}
\mu(v(0)>y)=\exp\left\{-\frac{y^\rho}{\rho}\right\}.
\end{equation}
For any function $f$ of the
potential, we will use the notations $\langle f\rangle:=\int fd\mu$,
and $Var_\mu(f):=\langle (f-\langle f\rangle)^2\rangle$
whenever they are well defined.
Let us also introduce the {\it conjugate exponent} $\rho'$ of $\rho>1$,
defined by the equation

$$
\frac{1}{\rho'}+\frac{1}{\rho}=1,
$$
and the notation $f\sim g$ to indicate that $\lim_{t\to\infty}f(t)/g(t)=1$.
Our first result gives the precise asymptotic behavior
of the average of the expectation of the total number of
random walks.

\medskip

\begin{theorem} 
\label{theorem4}
Consider the solution of (\ref{i1}) and (\ref{i1.111}).
Then, 

\begin{equation}
\label{annealed-asymp}
\langle m(0,t)\rangle\sim \left(\frac{\pi}{\rho-1}\right)^{1/2}
 t^{1-\frac{\rho'}{2}} e^{\frac{t^{\rho'}}{\rho'}-2d(\kappa t-t^{2-\rho'})}.
\end{equation}

\end{theorem}

\medskip

\noindent The average $\langle m(0,t)\rangle$ of Theorem \ref{theorem4}
will turn out to be the adequate centering of the empirical average
of the field $\{m(x,t:x\in\mathbb Z^d\}$ in certain regimes, leading
to the appearance of stable laws. To state the corresponding results
we still need to introduce additional notation.

Consider on ${\mathbb Z}^d$
the norm $||x||:=\sup\{|x_i|:1\le i\le d\}$.
For each $r\ge 0$ and $x\in{\mathbb Z}^d$
 consider the subset 
$\Lambda(x,r):=\{y\in{\mathbb Z}^d:||y||\le r\} $.
For $L>0$, we will use the notation $\Lambda_L$ instead
of $\Lambda(0,L)$. We now define the
averaged first moment at scale $L$ as

$$
m_L(t):=\frac{1}{|\Lambda_L|}\sum_{x\in\Lambda_L}m(x,t).
$$
We will say that a sequence $\gamma=\{z_1,z_2,\ldots ,z_p\}$ in
${\mathbb Z}^d$ is a {\it path} if for each $1\le j\le p-1$,
the sites $z_j$ and $z_{j+1}$ are nearest neighbors.
 The
{\it length} of  a path $\gamma$, denoted by $|\gamma|$, is
 equal to the number $p$ of sites defining it. Furthermore,
we will say that  $\gamma$ connects sites $x$ and $y$
if $z_1=x$ and $z_p=y$.
Denote the set of paths contained
 in a set $U\subset{\mathbb Z}^d$ and
connecting $x$ to $y$,  by ${\mathbb P}_{U}(x,y)$;
 the set of 
paths starting from $x$ and contained in $U$ by
${\mathbb P}_U(x)$. 
For each $n\ge 1$, define ${\mathbb P}_{U}^n(x,y):=\{\gamma\in
{\mathbb P}_{U}:|\gamma |=n\}$. Furthermore, to each path
$\gamma\in{\mathbb P}^n_U(0)$ we can associate a set
$\{y_1,\ldots ,y_k\}$, which represents the different
sites visited by the path, and a set $\{n_1,\ldots, n_j\}$
such that $\sum_{i=1}^k n_i=n$, such that $n_i$ represents
the number of times the site $y_i$ was visited by $\gamma$.
Finally, whenever $U={\mathbb Z}^d$
we will use the notation ${\mathbb P}(x)$, ${\mathbb P}(x,y)$,
${\mathbb P}^n(0)$,
instead of ${\mathbb P}_U(x)$, ${\mathbb P}_U(x,y)$,
and ${\mathbb P}^n(0)$,
respectively.
For each $v\in W$ and natural $N$, consider the function
defined for $s\ge 0$ as

\begin{equation}
\label{bnsv}
B_N(s,v):=
\frac{s+2d\kappa}{1+\sum_{j=1}^N
\sum_{\gamma\in{\mathbb P}^{2j+1}(0,0)}
\prod_{z\in\gamma,z\ne z_1} \frac{\kappa}{2d\kappa+(s-v_0(z))_+}},
\end{equation}
where $v_0(z):=v(z)$ for $z\ne 0$ while $v_0(0)=0$.
Also, define the constants

\begin{equation}
\label{mdef}
M:=\min\{j\ge 1:2j\rho' > 2(j+1)\},
\end{equation}

\begin{equation}
\label{mathcalb}
\mathcal B_N:=
\frac{2d\kappa}{1+\sum_{j=1}^N
\sum_{\gamma\in{\mathbb P}^{2j+1}(0,0)}
\left(\frac{1}{2d}\right)^{2j}
   },
\end{equation}

$$
\mathcal A_0=\mathcal A_0(\gamma,\rho):=
\left(\frac{\gamma\rho}{\rho'}\right)^{1/\rho},
$$

\begin{equation}
\label{stable-exponent}
\alpha=\alpha(\gamma,\rho):=\left(\frac{\gamma\rho}{\rho'}\right)^{1/\rho'}
\end{equation}
and
\begin{equation}
\label{gammas}
\gamma_1:=\frac{\rho'}{\rho}\qquad {\rm and}\qquad \gamma_2:=\frac{\rho'}{\rho} 2^{1/\rho'}.
\end{equation}
We also will need to define for each natural $N$ the function $\zeta_N(s):[0,\infty)\to \left(0,e^{-\frac{1}{\rho}\mathcal B_N^\rho}\right)$
by

\begin{equation}
\label{phifunct}
\zeta_N(s):=E_\mu\left[e^{-\frac{1}{\rho} B_N(s,v)^\rho}\right].
\end{equation}

\noindent Our second result establishes the existence of a function $h(t)$, which we will
call the {\it scaling function} and which provides the adequate scaling factor
$e^{th(t)}$ which will give the limiting stable laws.
To state it we need to define  the function $\tau:[1,\infty)\to\mathbb R$
by

\begin{equation}
\label{taudef}
\tau(L):=\frac{d\rho'}{\gamma}\log L.
\end{equation}

\medskip

\smallskip

\begin{theorem} 
\label{theorem3.1} 
Let  $0<\gamma<\gamma_2$, $L(t)$ an increasing function and
$\tau(t)=\tau(L(t))$ defined in (\ref{taudef}).
Then, the following statements are satisfied.

\begin{itemize}

\item[(i)] There exists a unique function $h(t):[t_0,\infty)\to [0,\infty)$
where $t_0$ is defined by $\frac{1}{L(t_0)^d}=e^{-\frac{1}{\rho}\mathcal B_M^\rho}$
[c.f. (\ref{mathcalb})] and  $M$  in (\ref{mdef}), which satisfies
the equation

\begin{equation}
\label{bheq}
\zeta_M(h(t))=\frac{1}{L(t)^d}.
\end{equation}

\item[(ii)]  The solution $h(t)$ of (\ref{bheq}) of part $(i)$ admits the expansion

\begin{equation}
\label{gexpansion}
h(t)=\mathcal A_0 \tau^{\rho'-1}(t)-2d\kappa +\sum_{1\le j\le M} h_j(t)
+O \left(\frac{1}{t^{(2M+1)\frac{\rho'}{\rho}}}\right),
\end{equation}
where $h_j(t)$, $1\le j\le M$ are functions recursively defined as

\begin{equation}
\label{hzero}
h_0(t):=\mathcal A_0\tau^{\rho'-1}(t)-2d\kappa,
\end{equation} 
and for $1\le j\le M$

\begin{equation}
\label{hformula}
h_j(t):=\frac{1}{\mathcal A_0^{\rho-1}}\left(\frac{\gamma}{\rho'}\tau^{\rho'-1}(t)
+\frac{1}{\tau(t)}\log E_\mu\left[e^{-\frac{1}{\rho} B_j(h_0+\cdots+h_{j-1},v)^\rho}
\right]\right)
\end{equation}
and

\begin{equation}
\label{gorder}
h_j(t)=O \left(\frac{1}{t^{(2j-1)\frac{\rho'}{\rho}}}\right).
\end{equation}

\end{itemize}
\end{theorem}

\noindent Throughout, for $\alpha\in (0,2)$ we will call $S_\alpha$
the distribution of the totally asymmetric stable law of
exponent $\alpha$ (and skewness parameter $1$) with characteristic
function

\begin{equation}
\label{charstable}
\phi_\alpha(u):=
\begin{cases}
\exp\left\{-\Gamma(1-\alpha)|u|^\alpha e^{-\frac{i\pi\alpha 
 }{2}{\rm sgn}(u)}\right\}&
{\rm if}\ 0<\alpha<1\\
\exp\left\{iu(1-\bar\gamma)-\frac{\pi}{2}|u|(1+ i\ {\rm sgn}(u)\cdot\frac{2}{\pi}\log |u| )\right\} &
{\rm if}\ \alpha=1\\
\exp\left\{\frac{\Gamma(2-\alpha)}{\alpha-1}|u|^\alpha e^{-\frac{i\pi\alpha 
 }{2} {\rm sgn}(u)}\right\}&
{\rm if}\ 1<\alpha<2,
\end{cases}
\end{equation}
where $\Gamma(s):=\int_0^\infty x^{s-1}e^{-x}dx$ is the gamma function,
${\rm sgn}(u):=\frac{u}{|u|}$ for $u\ne 0$ and ${\rm sgn}(u):=0$
for $u=0$ and $\gamma=0.5772...$ is the Euler constant.

 As the following result shows, the scaling function
of Theorem 1 appears as the right choice to rescale the empirical
average $m_L$ of the potential field, in order to obtain stable
limiting laws.

\medskip
 
\begin{theorem} 
\label{theorem3.2} 
Consider a potential $v$ having a Weibull distribution of parameter
$\rho>1$. Consider an increasing scale $L$ and let $\tau(t)=\tau(L(t))$ as
defined in (\ref{taudef}). We then  have that

\begin{equation}
\nonumber
\lim_{t\to\infty}\frac{ m_L(t)-A(t)}{e^{ th(t)}|\Lambda_L|^{-1}}=S_\alpha,
\end{equation}
where

\begin{equation}
\label{atdef}
A(t):=
\begin{cases}
0\quad & {\rm if}\quad 0<\gamma<\gamma_1\\
\left\langle m(0,t), \sum_{x\in\Lambda_t}m(x,t)\le e^{th(t)}
\right\rangle\quad & {\rm if}\quad \gamma=\gamma_1\\
\langle  m(0,t)\rangle\quad & {\rm if}\quad \gamma_1<\gamma<\gamma_2,
\end{cases}
\end{equation}
and where the convergence  is in distribution and 
$\alpha=\alpha(\gamma,\rho)$ is given by (\ref{stable-exponent}).

\end{theorem}

\medskip

An immediate consequence of Theorem \ref{theorem1},
is the appearance of a transition mechanism in the asymptotic
behavior of the quenched-annealed transition, where as the
value of $\rho$ increases in $(1,\infty)$, the number of terms
in the asymptotic expansion of the scaling function $h$ which
are relevant also increase. for even values of
$\rho$. Indeed, we have the following
corollary of Theorem \ref{theorem1}, which also shows that
the first terms in the expansion of the scaling function can
be computed explicitly.

\medskip

\begin{corollary}
\label{corolary0} Under the assumptions of Theorem \ref{theorem3.1},
for $0<\gamma<\gamma_2$, the scaling function $h(t)$ defined
in (\ref{bheq}) admits the expansion:

\begin{itemize}

\item[(i)] for  $1<\rho<2$,
 $ h(t)=\mathcal A_o \tau^{\rho'-1}-2d\kappa$,

\item[(ii)] for $2\le\rho<3$,
$ h(t)=\mathcal A_o \tau^{\rho'-1}-2d\kappa
+2d
 \kappa^2\mathcal A_0^3 \tau^{1-\rho'}$ and

\item[(iii)] for $3\le\rho<\frac{3+\sqrt{17}}{2}$,

\begin{eqnarray*}
&h(t)=\mathcal A_o \tau^{\rho'-1}-2d\kappa
+K_2 \tau^{1-\rho'}+K_3\tau^{\frac{\rho}{\rho-1}(3-2\rho')-1}
+K_4\frac{1}{\tau}\\
&+K_5
\log\tau+
o\left(\frac{1}{t}\right),
\end{eqnarray*}
where

$$
K_2:=\mathcal A_0^{2-\rho}K_1,\quad K_3:=\frac{K_1(\mathcal A_0K_1)^{\frac{1}{\rho-1}}}{\mathcal A_0^{\rho-1}}\left(1-\frac{1}{\rho}\mathcal A_0\right),
$$

$$
K_4:=(3-2\rho')\log \mathcal A_0+\frac{1}{2}\log\pi (\rho-1)-\mathcal A_0^{3-\rho}K_1^2\quad {\rm and}\quad K_5:=\frac{\rho}{2(\rho-1)}.
$$
\end{itemize}
\end{corollary}

\medskip

For $\rho$ larger
than $(3+\sqrt{17})/2$  
 extra terms in the expansion of $h$ have to be
computed in Corollary \ref{corolary0}.
 In order to keep the length of the computations limited,
we have decided to stop there. There seems to be no straightforward
interpretation on the appearance of this number.

 The proof of Theorem \ref{theorem1} is based on   rank-one
 perturbation
methods  to obtain asymptotic expansions of  the largest
eigenvalues of the Laplacian operator in a potential having
high peaks.

Throughout this article, a constant $C$ will always denote
a non-random number, independent of time. We will use
the letters $C$ or $C_1,C_2,\ldots$ to denote
them. We will
use the notation $O(s):[0,\infty)\to [0,\infty)$
 to denote a function (possibly random)
which satisfies for all $s\ge 0$,

$$
|O(s)|\le C s
$$
for some constant $C$.

\medskip

\section{Rank-one perturbation for Schr\"odinger operators} 
\label{secro}
Here we will apply  perturbation theory
to study the asymptotic behavior of the principal Dirichlet eigenvalue and
 eigenfunction
of the Laplacian operator plus a potential with
a rank-one perturbation. The results we will present
are deterministic, in the sense that we do not assume
that the potential is random, and are not particularly original, since
they correspond to a standard application of rank-one perturbation theory
(see for example \cite{simon}).
With the aim of giving a self-contained presentation, the basic tools
 of rank-one perturbation theory that will be used are
presented in Appendix \ref{fr}.
In subsection \ref{resultsc}, we give
a precise statement about the principal Dirichlet eigenvalue and
eigenfunction of the perturbed operator in Theorem \ref{eigenvf}.
 In subsection \ref{preliminc} we
derive some estimates and formulas about the spectrum of the
unperturbed Schr\"odinger operator
and its Green function. In subsection \ref{proofc} we relate
the results of subsection \ref{preliminc} to the spectrum of
the perturbed Schr\"odinger operator through the rank-one perturbation
theory presented in Appendix \ref{fr}, to prove Theorem \ref{eigenvf}.

\smallskip

\subsection{Principal eigenvalue and eigenfunction}
\label{resultsc}
Let $U$ be a finite connected subset of ${\mathbb Z}^d$.
Consider the {\it Schr\"odinger operator}

\begin{equation}
\nonumber
H^0_{U,w}:=\kappa\Delta_U +w,
\end{equation}
on $U$, with Dirichlet boundary conditions, where $\kappa>0$
and $w$ is a {\it non-negative potential} on $U$:
a  set $w=\{w(x):x\in U\}$, where $w(x)\ge 0$.
In other
words, $H^0_{U,w}$ is the operator defined on $l^2(U)$
acting as,

$$
H^0_{U,w} f(x)=\kappa\sum_{j=1}^{2d}(f(x+e_j)-f(x))+w(x)f(x),\qquad x\in U,
$$
with the convention that $f(y)=0$ if $y\notin U$,
and where $\{e_j:1\le j\le 2d\}$ are the canonical generators 
of ${\mathbb Z}^d$
and their corresponding inverses.
Note that $H^0_{U,w}$ is a bounded symmetric operator on $l^2(U)$.
Define 

$$
\bar w_U:=\max\{w(x):x\in U\}.
$$

\medskip

\begin{theorem}
\label{eigenvf} Let $U\subset \mathbb Z^d$ and $w$ a potential
on $U$ and $x_0\in U$, and assume that $w(x_0)=0$.
 Consider the Schr\"odinger operator $H^0_{U,w}$. Then,
whenever $h>\bar w_U$,

$$
H_{U,w,h}:=H^0_{U,w}+h\delta_{x_0}
$$
has a simple principal Dirichlet eigenvalue $\lambda_0$
and a principal Dirichlet eigenfunction $\psi_0$. Furthermore,
the following are satisfied.

\begin{itemize}

\item[(i)] The principal Dirichlet eigenvalue has the expansion

\begin{equation}
\label{lambdaexps}
\lambda_0=h-2d\kappa+h\sum_{j=1}^\infty
\sum_{\gamma\in{\mathbb P}^{2j+1}_U(x_0,x_0)}
\prod_{z\in\gamma,z\ne z_1} \frac{\kappa}{\lambda_0+2d\kappa-w(z)}.
\end{equation}

\item[(ii)] There is a constant $K$ such that
for all $x\in U$ one has that

$$
 \psi_0(x)=K
\sum_{\gamma\in
 {\mathbb P}_U(x,x_0)}\prod_{z\in\gamma}
\frac{\kappa}{\lambda_0+2d\kappa - w(z)}.
$$

\item[(iii)] Whenever $h\ge 4d\kappa $ we have that

$$
\psi_0(x)=\mathbbm 1_{x_0}(x)+\varepsilon(x),
$$
where $\varepsilon(x)$ satisfies for all $x\in U$

$$
|\varepsilon(x)|\le C\frac{1}{(h-\bar w_U)^{|x-x_0|_1+1}},
$$
for some constant $C$ that does not depend on $h$ nor $w$.

\item[(iv)] We have that

$$
\sup\{\lambda\in\sigma(H_{U,\omega,h}):\lambda< \lambda_0\}\le\bar w_U.
$$
\end{itemize}
\end{theorem}

\medskip

\noindent Theorem \ref{eigenvf} will be proved in subsection \ref{proofc}.

\subsection{Green function of the unperturbed operator}
\label{preliminc} Throughout,
given any self-adjoint operator $A$ defined on
$l^2(U)$, we will denote by $res(A)$ and $\sigma(A)$ the resolvent
set and the spectrum of $A$ respectively. For $\lambda\in res (H^0_{U,w})$, let us introduce on
$U\times U$ the function

$$
g^{U,w}_\lambda(x,y):=(\delta_{x},(\lambda I-H^0_{U,w})^{-1}\delta_{y}).
$$
Note that $g^{U,w}_\lambda$ is
the {\it Green function} of the operator $H^0_{U,w}-\lambda I$.

Note that,

\begin{equation}
\label{eu2}
{\mathbb P}_{U}(x,y)=\cup_{k=1}^\infty 
{\mathbb P}^k_{U}(x,y),
\end{equation}
where the union is disjoint.
Let us now introduce the
 norm $|x|_1:=\sum_{j=1}^d |x_j|$ on ${\mathbb Z}^d$.
Note that if  $\gamma\in{\mathbb P}_U(x,y)$,
then

\begin{equation}
\label{eu1}
|\gamma|\ge |x-y|_1+1.
\end{equation}
Furthermore

\begin{equation}
\label{eu3}
|{\mathbb P}^k_U(x,y)|\le (2d)^{k-1}.
\end{equation}
\smallskip

\begin{lemma}
\label{leu1}  Let $U$ be a bounded connected subset of ${\mathbb Z}^d$, $w$ a
non-negative potential on $U$ and $\kappa>0$. Let $x,y\in U$. 
Then, the Green function $g^{U,w}_\lambda(x,y)$ is analytic if 
$\lambda\in res (H^0_{U,w})$ and

\begin{equation}
\label{eu4}
g^{U,w}_\lambda (x,y)=\frac{1}{\kappa}\sum_{\gamma\in
 {\mathbb P}_U(x,y)}\prod_{z\in\gamma}
\frac{\kappa}{\lambda+2d\kappa - w(z)}.
\end{equation}

\end{lemma}

\begin{proof} 
Now, let us remark that for $x,y\in{\mathbb Z}^d$,

\begin{eqnarray}
\nonumber
&g^{U,w}_\lambda(x,y)=\int_0^\infty E_x\left[e^{\int_0^t (w(X_s)-\lambda)ds}\delta_y(X_t)\right]
dt\\
\label{prf1}
&=\sum_{\gamma\in{\mathbb P}_U(x,y)}\int_0^\infty E_x\left[
e^{\int_0^t (w(X_s)-\lambda)ds} C_\gamma\right]dt,
\end{eqnarray}
where $\{X_t:t\ge 0\}$ is a simple symmetric random walk of total jump
rate $2d\kappa$ starting from $x$ and $E_x$ the corresponding
expectation. For each $\gamma=\{x_1=x,\ldots, x_n=y\}\in{\mathbb
P}_U(x,y)$ define $B_\gamma:= E_x\left[ e^{\int_0^t
(w(X_s)-\lambda)ds} C_\gamma\right]$.  Then, noting that the
probability density on the path $\gamma$ is given by $e^{-2d\kappa
s_1-\cdots-2d\kappa s_{n-1} -2d\kappa(t-s_1-\cdots -s_{n-1})}$, where
$s_i$ is the time spent on $x_i$, we see that,

\begin{eqnarray*}
&B_\gamma=
\frac{(2d\kappa)^{n-1}}{(2d)^{n-1}}
e^{-2d\kappa t}
\int_{S_{n-1}}
e^{\bar w(x_1)s_1+\cdots+\bar w(x_{n-1})s_{n-1}+\bar w(x_n)(t-s_1-\cdots-s_{n-1})}ds_1\cdots ds_{n-1}\\
&
=\kappa^{n-1}
e^{-2d\kappa t+\bar w_n t}
\int_{S_{n-1}} e^{(\bar w_1-\bar w_n)s_1+\cdots(\bar w_{n-1}-\bar w_n)s_{n-1}}ds_1\cdots ds_{n-1},
\end{eqnarray*}
where $\bar w_i=w(x_i)-\lambda$ and $S_{n-1}=\{s_1+\cdots +s_{n-1}<t\}$. Using induction on $n$
we can compute the above integral to obtain,

$$
B_\gamma=\frac{1}{\kappa}\prod_{z\in\gamma}\frac{\kappa}{\lambda+2d\kappa-w(z)}.
$$
Substituting this expression back in (\ref{prf1}) finishes the proof.
\end{proof}

\smallskip

\noindent Note that the largest eigenvalue of $H_{U,w}$
can be expressed as $\lambda^{U,w}_+:=\sup\{\lambda\in\sigma (H_{U,w})\}$.

\smallskip

\begin{lemma}
\label{leu2} For every finite connected set $U$ and non-negative
potential $w$ on $U$.

\begin{equation}
\label{eu5}
\bar w_U-2d\kappa\le \lambda^{U,w}_+\le\bar w_U.
\end{equation}
\end{lemma}

\begin{proof} Let $x_m$ be some site where  $\bar w_U=w(x_m)$.
 The first inequality of (\ref{eu5}) follows from the computation
$(f,H_{U,w}f)\ge w(x_m)-2d\kappa$ for
$f=\delta_{x_m}$.  The second
from the estimate $(f,H_{U,w} f)\le w(x_m)$, for arbitrary
$f\in l^2(U)$ with unit norm.
\end{proof}
\smallskip

\begin{corollary}
\label{c1} 
Let $x,y\in U$.
If $\lambda-\lambda^{U,w}_+>2d\kappa$, then

\begin{equation}
\label{eu12}
\frac{1}{\kappa}
\left(\frac{\kappa}{\lambda+2d\kappa}\right)^{|x-y|_1+1}\le g^{U,w}_\lambda (x,y)\le \frac{(2d\kappa)^{|x-y|_1}}{(\lambda-\lambda^{U,w}_+)^{|x-y|_1}}\frac{1}{\lambda-\lambda^{U,w}_+-2d\kappa}.
\end{equation}

\end{corollary}
\begin{proof}
For every $z\in U$,
$\lambda+2d\kappa\ge \lambda+2d\kappa - w(z)$. By part $(i)$ of
lemma \ref{leu1} and the fact that the shortest path
between $x$ and $y$ has length $|x-y|_1+1$, we prove the
first inequality
of display (\ref{eu12}).
Also, using again part $(i)$ of lemma \ref{leu1}, inequality (\ref{eu1}) and the decomposition
(\ref{eu2})
 we see that,

$$
g^{U,w}_\lambda (x,y)\le\frac{1}{\kappa} \sum_{k=|x-y|_1+1}^\infty\sum_{\gamma\in {\mathbb P}^k_U(x,y)}\left(
\frac{\kappa}{\lambda-\lambda^{U,w}_+}\right)^k.
$$
Now, from (\ref{eu3}), we conclude the proof.

\smallskip
\end{proof}

\smallskip

\noindent We can now derive the following lemma.

\smallskip

\begin{lemma}
\label{lgf1} Assume that $h>\bar w_U+2d\kappa$. Then,
there exists a unique $\lambda_0>\lambda^{U,w}_+$, which
satisfies the equation,

\begin{equation}
\label{gf1}
h g^{U,w}_{\lambda_0}(x,x)=1.
\end{equation}

\end{lemma}

\begin{proof} By Lemma \ref{es}, note that equation (\ref{gf1})
has a unique solution $\lambda_0>\lambda^{U,w}_+$ if,

$$
h>\frac{1}{\lim_{\lambda\searrow\lambda^{U,w}_+} g^{U,w}_\lambda(x,x)}.
$$
Now, by the first inequality of (\ref{eu12}) of part $(i)$ of
corollary \ref{c1} and by lemma \ref{leu2} the right-hand side of the above inequality is larger
than $\bar w_U+2d\kappa$.
\end{proof}

\subsection{Proof of Theorem \ref{eigenvf}}
\label{proofc} Let us now prove Theorem \ref{eigenvf}.
To prove part $(i)$ note that the unique $\lambda_0$ which
satisfies (\ref{gf1}) of Lemma \ref{lgf1} has to be the principal
Dirichlet eigenvalue by Theorem \ref{akform} of Appendix A.
Applying the expansion
(\ref{eu4}) of Lemma \ref{leu1} of the Green function we obtain
part $(i)$. To prove part $(ii)$ let us first note that
by Lemma \ref{lgf1}, $\lambda_0$ is an isolated point of the spectrum of
$H_{U,\omega}$. It follows that there is simple closed curve $\Gamma$ in the complex
plane which contains $\lambda_0$ in its interior and the rest of the spectrum of $H_{U,\omega}$
in its exterior. Therefore

$$
P:=\int_{\Gamma} R_\lambda d\lambda,
$$
where $R_\lambda$ is the resolvent of $H_{U,\omega}$,
is the orthogonal projection onto the eigenspace of $\lambda_0$. Now, from Theorem 5 of
Appendix A, we can see that 

$$
P=C_3 (q_{\lambda_0},\cdot)q_{\lambda_0},
$$
for some constant $C_3$. From here we can deduce part $(ii)$.
Part $(iii)$ follows immediately from part $(ii)$.
Part $(iv)$ follows from Corollary \ref{t1} of Appendix \ref{fr} and
Lemma \ref{leu2}.

\medskip

\section{High  peak statistics of a Weibull potential}
\label{secwex} 
In this section we will derive several results describing the
asymptotic behavior as $l\to\infty$ of quantities defined in
terms of the order statistics $v_{(1)}\ge v_{(2)}\ge\cdots
\ge v_{(N)}$ of an  i.i.d. potential $v$ on the
box $\Lambda_l$ with Weibull distribution.  We will ocasionally also use the notation
$v(x_{(j)})$ instead of $v_{(j)}$ to indicate explicitly the
site $x_{(j)}$ where the value $v_{(j)}$ is attained.

\smallskip
\begin{lemma}
\label{lemma1}
 Consider the order statistics
$\{v_{(1)},\ldots ,v_{(N)}\}$ of the potential $v$.
Let 

$$
a_l:=(\rho\log |\Lambda_l|)^{1/\rho}.
$$
\begin{itemize}
\item[(i)]  For every $x\in{\mathbb R}$,

\begin{equation}
\nonumber
\lim_{l\to\infty}\mu[  v_{(1)}< a_l+b_l x]=\exp\left\{-e^{- x}\right\},
\end{equation}
where
$b_l:=\frac{1}{(\rho\log |\Lambda_l|)^{1-1/\rho}}$. In particular, $\mu$-a.s.,

\begin{equation}
\nonumber
\lim_{l\to\infty}\frac{ v_{(1)}}{a_l}=1.
\end{equation}

\item[(ii)] For every sequence $\{c_l\}$ such that $c_l\ge a_l$,
we have that

\begin{equation}
\label{boundv1}
\mu[ v_{(1)}\ge c_l]\le
|\Lambda_l|e^{-\frac{1}{\rho}c_l^\rho}+o\left(|\Lambda_l|e^{-\frac{1}{\rho}c_l^\rho}\right).
\end{equation}
and that

\begin{equation}
\label{boundv2}
\mu[ v_{(2)}\ge c_l]\le
|\Lambda_l|^4e^{-\frac{2}{\rho}c_l^\rho}
\end{equation}
\end{itemize}
\end{lemma}
\begin{proof} 
Consider a sequence $\{c_l:l\ge 0\}$ and note that

\begin{equation}
\nonumber
\mu(v_{(1)}<c_l)=\left(1-\mu(v_1\ge c_l)\right)^{|\Lambda_l|}
= \left(1-\frac{1}{|\Lambda_l|^{\left(\frac{c_l}{a_l}\right)^\rho}}\right)^{|\Lambda_l|}.
\end{equation}
Therefore, using the fact the for all natural $n$, $e^{-1}(1-1/n)\le (1-1/n)^n\le e^{-1}$, we have that

\begin{equation}
\label{vineq}
\exp\left\{-|\Lambda_l|^{1-\left(\frac{c_l}{a_l}\right)^\rho}\right\}
\left(1-\frac{1}{|\Lambda_l|^{\left(\frac{c_l}{a_l}\right)^\rho}}
\right)^{|\Lambda_l|^{1-\left(\frac{c_l}{a_l}\right)^\rho}}
\le
\mu(v_{(1)}<c_l)\le \exp\left\{-|\Lambda_l|^{1-\left(\frac{c_l}{a_l}\right)^\rho}\right\}.
\end{equation}
Choosing $c_l=a_l+b_l x$ in the above inequalities and taking the limit when $l\to\infty$, we prove part $(i)$.

\smallskip
\noindent {\it Part (ii).} Note that for all $x\ge 0$, $e^{-x}\le 1-x+\frac{x^2}{2}$. 
Therefore, by (\ref{vineq}), we have that

\begin{eqnarray*}
&\mu(v_{(1)}\ge c_l)\le\left(
 \frac{1}{|\Lambda_l|^{\left(\frac{c_l}{a_l}\right)^\rho-1}}- \frac{1}{|\Lambda_l|^{2\left(\frac{c_l}{a_l}\right)^\rho-2}}
\right)\left(1-\frac{1}{|\Lambda_l|^{\left(\frac{c_l}{a_l}\right)^\rho}}
\right)^{|\Lambda_l|^{1-\left(\frac{c_l}{a_l}\right)^\rho}}\\
&=|\Lambda_l|e^{-\frac{1}{\rho}c_l^\rho}+o\left(|\Lambda_l|e^{-\frac{1}{\rho}c_l^\rho}\right),
\end{eqnarray*}
which gives (\ref{boundv1}). The proof of (\ref{boundv2}) is completely similar.
\end{proof}
\smallskip

\section{Stable limit laws and structure of the scaling}
\label{secpt}

Here we will prove Theorem \ref{theorem3.1}.
Part $(i)$ will be proved in subsection \ref{esf} and
part $(ii)$ in subsection \ref{tprop}.

\medskip

\subsection{Existence of the scaling function}
\label{esf}
Here we will prove part $(i)$ of  Theorem \ref{theorem3.1},
showing
the existence of a function $h(t)$ which satisfies the equality (\ref{bheq}).

\medskip

\begin{lemma}
\label{homeo}
 For every potential $v\in W$ and
natural $N$
the function $\zeta_N(s):[0,\infty)\to (0,e^{-\frac{1}{\rho}\mathcal B_M^\rho}]$
[c.f. (\ref{mathcalb})] defined by (\ref{phifunct})
is a homeomorphism.
\end{lemma}
\begin{proof} 
Note that for every fixed potential $v\in W$ and natural $N$ the function $B_N(s,v)$ defined in (\ref{bnsv})
is strictly increasing in $[0,\infty)$. By the bounded
convergence theorem, this implies that $\phi$ is
strictly decreasing and continuous with range $(0,
e^{-\frac{1}{\rho}\mathcal B_M^\rho}]$, which proves the lemma.
\end{proof}

\medskip

\noindent Now note that by the definition of  $t_0$
given in part $(i)$ of Theorem \ref{theorem3.1}, since the scale $L(t)$ is an
increasing function of $t$, we have that whenever $t\ge t_0$

$$
0<e^{-\gamma \frac{\tau(t)^{\rho'}}{\rho'}}\le e^{-\frac{1}{\rho}\mathcal B_M^\rho}.
$$
By Lemma \ref{homeo} it is clear that  for each $t\ge t_0$
there exists a $h(t)\in [0,\infty)$ such that (\ref{bheq})
is satisfied. We define now as in the statement of Theorem \ref{bheq} for $t\ge 0$,

\begin{equation}
\label{gdefi}
g(t):=t h(t).
\end{equation}

\medskip

\subsection{Properties of the scaling function}
\label{tprop}
Here we will prove part $(ii)$ of Theorem \ref{theorem3.1},
which states that the scaling function $h(t)$
defined in (\ref{bheq}) has the expansion specified by
(\ref{gexpansion}), (\ref{hformula}) and (\ref{gorder}).

We will now prove that the functions
defined recursively by (\ref{hzero}) and (\ref{hformula}) of Theorem \ref{theorem1} are such that (\ref{gexpansion}) and (\ref{gorder})
are satisfied. Define for $1\le j\le M$, $\mathcal H_j:=h_0+\cdots+h_j$. 
We will need the following lemma.

\medskip

\begin{lemma}
\label{lemmanov}
 For every $\tau>0$ the following are satisfied.

\begin{itemize}

\item[(i)]  For all $\epsilon>0$
and $\max_{e\in E} v(e)\le (1-\epsilon)\mathcal A_0\tau^{\rho'-1}$,
we have that

\begin{equation}
\label{fin-1}
B_1(h_0,v)^\rho=\mathcal A_0^\rho \tau^{\rho'}+O\left(\frac{1}{\tau^{\rho'-2}}\right),
\end{equation}
and for $j\ge 1$ that

\begin{equation}
\label{fin-2}
B_{j+1}(\mathcal H_j,v)^\rho=B_{j}(\mathcal H_{j-1},v)^\rho
+\rho\mathcal A_0 \tau h_j
+O\left(\frac{1}{\tau^{(2j+1)\rho'-(2j+2)}}\right).
\end{equation}
where both in (\ref{fin-1}) and (\ref{fin-2}) the error term satisfies
for all real $x$,
$$
\left|O\left(x\right)\right|\le C_8 |x|.
$$
\item[(ii)] There is a constant $C_{6,1}$ such that

\begin{equation}
\label{fin0}
e^{-\left(\gamma\frac{\tau^{\rho'}}{\rho'}+\frac{C_{6,1}}{\tau^{\rho'-2}}\right)}\le
E_\mu\left[e^{-\frac{1}{\rho}B_1(h_0,v)^\rho}\right]\le
e^{-\left(\gamma\frac{\tau^{\rho'}}{\rho'}-\frac{C_{6,1}}{\tau^{\rho'-2}}\right)}.
\end{equation}
Similarly, for each $j\ge 1$ there is a constant $C_{6,j}$ such that

\begin{eqnarray}
\nonumber
&E_\mu\left[e^{-\left(\frac{1}{\rho}B_j(\mathcal H_{j-1},v)^\rho+\rho\tau h_j+\frac{C_{6,j}}{\tau^{(2j+1)\rho'-(2j+2)}}\right)}\right]\le
E_\mu\left[e^{-\frac{1}{\rho}B_{j+1}(\mathcal H_{j},v)^\rho}\right]\\
\label{fin01}
&\le
E_\mu\left[ e^{-\left(\frac{1}{\rho}B_j(\mathcal H_{j-1},v)^\rho+\rho\tau h_j-\frac{C_{6,j}}{\tau^{(2j+1)\rho'-(2j+2)}}\right)}\right].
\end{eqnarray}

\end{itemize}

\end{lemma}
\begin{proof} Part $(i)$ follows from a standard Taylor expansion.
We will now prove (\ref{fin0}) of part $(ii)$. Define for $\epsilon>0$ the event

$$
A:=\left\{\max_{e\in U}v(e)\le (1-\epsilon)\mathcal A_0 \tau^{\rho'}\right\}.
$$
Then

$$
E_\mu\left[e^{-\frac{1}{\rho}B_1(h_0,v)^\rho}\right]=
E_\mu\left[e^{-\frac{1}{\rho}B_1(h_0,v)^\rho}, A\right]+
E_\mu\left[e^{-\frac{1}{\rho}B_1(h_0,v)^\rho},A^c\right].
$$
But by part $(i)$ we have that

\begin{equation}
\label{fin1}
E_\mu\left[e^{-\frac{1}{\rho}B_1(h_0,v)^\rho}, A\right]\le
e^{-\left(\gamma\frac{\tau^{\rho'}}{\rho'}-\frac{C_8}{\tau^{\rho'-2}}\right)}.
\end{equation}
On the other hand, since 

$$
B_1(h_0,v)^\rho\ge \mathcal \tau^{\rho'}(\mathcal A_0 
\mathcal B_1)^\rho,
$$
 we have that for $\epsilon$ small enough

\begin{eqnarray}
\nonumber
&E_\mu\left[e^{-\frac{1}{\rho}B_1(h_0,v)^\rho}, A^c\right]\le
2de^{-\tau^{\rho'}\left((\mathcal A_0\mathcal B_1)^\rho+
\frac{1}{\rho}(1-\epsilon)^\rho\mathcal A_0^\rho\right)}\\
\label{fin2}
&=
2de^{-\tau^{\rho'}\left((\mathcal A_0\mathcal B_1)^\rho+
\frac{\gamma}{\rho'}(1-\epsilon)^\rho\right)}
=O\left(e^{-(1+\epsilon)\gamma\frac{\tau^{\rho'}}{\rho'}}\right).
\end{eqnarray}
Combining (\ref{fin1}) with (\ref{fin2}) we see that
there is a constant $C_6$ such that

\begin{equation}
\nonumber
E_\mu\left[e^{-\frac{1}{\rho}B_1(h_0,v)^\rho}\right]\le
e^{-\left(\gamma\frac{\tau^{\rho'}}{\rho'}-\frac{C_6}{\tau^{\rho'-2}}\right)}.
\end{equation}
On the other hand, from $B_1(h_0,v)^\rho\le\gamma\frac{\tau^{\rho'}}{\rho'}$, we
immediately get that

\begin{equation}
\nonumber
E_\mu\left[e^{-\frac{1}{\rho}B_1(h_0,v)^\rho}\right]\ge
e^{-\gamma\frac{\tau^{\rho'}}{\rho'}},
\end{equation}
which finishes the proof of (\ref{fin0}). The proof of (\ref{fin01}) is analogous.

\end{proof}

\medskip

\noindent Let us now prove (\ref{gorder}) of part $(ii)$ of Theorem \ref{theorem3.1}. By the definition of $h_1$ in (\ref{hformula})
and by (\ref{fin0}) of part $(ii)$ of Lemma \ref{lemmanov}, we see that

\begin{eqnarray}
\nonumber
&-\frac{C_{6,1}}{t^{\rho'-1}}\le
\frac{1}{\mathcal A_0^{\rho-1}}\left(
\gamma\frac{\tau^{\rho'-1}}{\rho'}
+\frac{1}{\tau}\log e^{-\left(\gamma\frac{\tau^{\rho'}}{\rho'}+\frac{C_{6,1}}{\tau^{\rho'-2}}
\right)}\right)\\
\nonumber
&\le
\frac{1}{\mathcal A_0^{\rho-1}}\left(
\gamma\frac{\tau^{\rho'-1}}{\rho' }
+\frac{1}{\tau}\log E_\mu\left[e^{-\frac{1}{\rho}B_1(h_0,v)^\rho}\right]
\right)=h_1(t)
\\
\nonumber
&\le
\frac{1}{\mathcal A_0^{\rho-1}}\left(\gamma\frac{\tau^{\rho'-1}}{\rho' }
+\frac{1}{\tau}\log e^{-\left(\gamma\frac{\tau^{\rho'}}{\rho'}-\frac{C_{6,1}}{\tau^{\rho'-2}}
\right)}\right)
\le
\frac{C_{6,1}}{t^{\rho'-1}}.
\end{eqnarray}
Hence,

 $$
h_1(t)=O\left(\frac{1}{t^{\rho'-1}}\right).
$$
 We will now
prove that for $j\ge 1$

\begin{equation}
\label{hjorder}
h_{j+1}(t)=O\left(\frac{1}{t^{(2j+1)(\rho'-1)}}\right).
\end{equation}
Using the definition (\ref{hformula}) for $h_j$ and  $h_{j+1}$, 
by (\ref{fin01}) of part $(ii)$ of Lemma \ref{lemmanov}
we conclude that

\begin{eqnarray}
\nonumber
&-\frac{C_{6,j-1}}{t^{(2j+1)(\rho'-1)}}\le
\frac{1}{\mathcal A_0^{\rho-1}}\left(\gamma\frac{\tau^{\rho'-1}}{\rho'}
+\frac{1}{\tau}\log E_\mu\left[ e^{-\left(\frac{1}{\rho}B_j(\mathcal H_{j-1},v)^\rho
+\rho\mathcal A_o\tau h_j
+\frac{C_{6,j-1}}{\tau^{(2j+1)\rho'-(2j+2)}}
\right)}\right]\right)\\
\nonumber
&\le
\frac{1}{\mathcal A_0^{\rho-1}}\left(\gamma\frac{\tau^{\rho'-1}}{\rho'}
+\frac{1}{\tau}\log E_\mu\left[ e^{-\frac{1}{\rho}B_{j+1}(\mathcal H_j,v)^\rho}\right]\right)=h_{j+1}(t)\\
\nonumber
&\le
\frac{1}{\mathcal A_0^{\rho-1}}\left(\gamma\frac{t^{\rho'-1}}{\rho'}
+\frac{1}{t}\log E_\mu\left[ e^{-
\left(\frac{1}{\rho}B_j(\mathcal H_{j-1},v)^\rho+\rho\mathcal A_o\tau h_j
-\frac{C_{6,j-1}}{\tau^{(2j+1)\rho'-(2j+2)}}
\right)}\right]\right)\le
\frac{C_{6,j-1}}{t^{(2j+1)(\rho'-1)}},
\end{eqnarray}
which proves (\ref{hjorder}) and hence (\ref{gorder}) of part $(ii)$ of Theorem \ref{theorem3.1}. It now follows that

$$
e^{-\gamma\frac{\tau^{\rho'}}{\rho'}}=
E_\mu \left[e^{-\frac{1}{\rho}B_M(\mathcal H_{M-1},v)^\rho-\mathcal A_0^{\rho-1}\tau
h_M}\right]
=
E_\mu \left[e^{-\frac{1}{\rho}B_M(\mathcal H_{M},v)^\rho+
O\left(\frac{1}{\tau^{(2M+1)\rho'-(2M+2)}}\right)
}\right]
$$
which implies that $h(t)-\mathcal H_M(t)=O\left(\frac{1}{t^{(2M+1)(\rho'-1)}}\right)$, which proves (\ref{gexpansion})
of part $(ii)$ of Theorem \ref{theorem3.1}.

\medskip

\section{Convergence to stable laws}
\label{tpart3} We will now prove
 Theorem \ref{theorem3.2}.
Some of the computations will be similar to those
done by Ben Arous, Bogachev and Molchanov in \cite{bbm05}, within
the context of sums of i.i.d. random exponential variables.
As a first step, we will first recall the coarse graining
methods introduced in \cite{bmr}, which will enable
us to reduce the problem to a sum of approximately
independent random variables in subsection \ref{messc}.
In subsection \ref{rind}, we will show how to reduce
the problem to an sum of independent random variables.
In subsection \ref{criteria} we recall a classical criteria 
in Theorem \ref{theorem-criteria} for
convergence to  infinite divisible distributions. 
These criteria will be verified in subsections \ref{clsf},
\ref{csigma} and \ref{cbeta}.

\medskip

\subsubsection{Mesoscopic scales}
\label{messc}
Let us now recall the coarse graining methods introduced in \cite{bmr}.
Let $L\ge 0$ and consider a box $\Lambda_L$. Here we will need the {\it strip-box
partition} of \cite{bmr}. We introduce a parameter $l$ smaller than
or equal to $L$, called the {\it mesoscopic scale}. Then, there exist
natural numbers $q$ and $\bar q$ such that $2L+1=ql+\bar q$, with
$0\le\bar q\le q$. Hence,

$$
2L+1=\sum_{i=1}^q l_i,
$$
where $l_i=l+\theta_{\bar q}(i)$ and $\theta_{\bar q}(i)=1$ for
$i\le\bar q$ and $\theta_{\bar q}(i)=0$ for $i>\bar q$.
Given a pair of real numbers $a,b$, we will use the notation
$[[a,b]]$ for $[a,b]\cap{\mathbb Z}$. Now define $I_1:=
[[-L,-L+l_1-1]]$ and for $1<i\le q$ let
$I_i:=\left[\left[-L+\sum_{j=1}^{i-1}l_i, -L+\sum_{j=1}^{i}l_i -1
\right]\right]$.
Now, we introduce a second parameter $r$ which is a natural number smaller
than or equal to $l$, called the {\it fine scale}. Let $r_i:=r
+\theta_{\bar q}(i)$. Define $J_1:=[[-L+r_1,-L+l_1-1-r_1]]$ and for
$1<i\le q$ let
$J_i:=
\left[\left[-L+\sum_{j=1}^{i-1}l_i+r_i, -L+\sum_{j=1}^{i}l_i -1-r_i\right]
\right]$.

Now, let ${\mathcal I}:=\{1,2,\ldots,q\}^d$.  For a given element
${\bf i}\in{\mathcal I}$, of the form ${\bf i}=(i_1,\ldots ,i_d)$
with $1\le i_k\le q$, $1\le k\le d$, we define,

$$
\Lambda_{\bf i}'':=J_{i_1}\times J_{i_2}\times\cdots\times J_{i_d},
$$
called a {\it main box}. Its cardinality is
$|\Lambda_{\bf i}''|=(l-2r)^d$. Now let,

$$
S_L:=\Lambda_L-\bigcup_{{\bf i}\in{\mathcal I}}\Lambda_{\bf i}'',
$$
called the {\it strip set}. The sets $S_L$ and $\{\Lambda_{\bf i}'':{\bf i}\in
{\mathcal I}\}$ define a partition of $\Lambda_L$ called the
{\it strip-box partition at scale $l$} of $\Lambda_L$.

Let us now write,

\begin{equation}
\label{mesodecomp}
\sum_{x\in\Lambda_L}m(x,t)=
\sum_{{\bf i}\in{\mathcal I} } m_{\bf i}+\sum_{x\in S_L}m(x,t),
\end{equation}
where

$$
m_{\bf i}:=\sum_{x\in\Lambda''_{\bf i}}m(x,t).
$$
We will use the notation ${\bf 0}:=\{0,\ldots,0\}$.
Throughout we will make the following choices

\begin{equation}
\label{scalel}
L:=e^{\gamma \frac{t^{\rho'}}{\rho'}},
\end{equation}

\begin{equation}
\label{scalelm}
l:=e^{ t}
\end{equation}
and

\begin{equation}
\label{scaler}
r:=t^2.
\end{equation}

\subsubsection{Dirichlet boundary conditions}
\label{rind} Throughout, we will denote
by $m(x,t,v)$ the solution $m(x,t)$ of the parabolic Anderson problem (\ref{i1}), 
emphasizing the dependence on the potential $v$ of it. For each finite set $U\subset{\mathbb
  Z}^d$ and environment $v\in W$, we define $\tilde m_U(x,t)= m_U (x,t,v_U)$ as the solution of
the parabolic Anderson equation with potential $v$ with Dirichlet boundary conditions on $U$
and initial condition ${\bf 1}_U$, so that

\begin{eqnarray}
\nonumber
&\displaystyle{ \frac{\partial \tilde m_U(x,t)}{\partial t}}=\kappa\Delta \tilde m_U(x,t)
+v(x)\tilde m_U(x,t),\quad {\rm for}\ {\rm all}\ t>0, x\in\mathbb Z^d,\\
\nonumber
&
\tilde m_U(x,0)={\bf }1_U(x),
\end{eqnarray}
where the Laplacian $\Delta$ has Dirichlet boundary conditions on $U$.
In other words,  $v_U(x)=v(x)$ for $x\in
U$ while $V_U(x)=-\infty$ for $x\notin U$, with $m(x,t,v_U)$ the solution of
(\ref{i1}). Throughout, for $r>0$, we will use the notation $\tilde m_r(x,t)$
instead of $\tilde m_{\Lambda(0,r)}(x,t)$. The following lemma
can be proved in the same way as 
part $(iii)$ of Proposition 9 of \cite{bmr2}.

\smallskip

\begin{lemma}
\label{lemma-truncation} Consider a finite subset $U\in{\mathbb Z}^d$. Then, for each
  $\beta>0$,
$\gamma>0$, there is a constant $C>0$ such that for all $R\ge 2\kappa t$,

$$
\left\langle |m(x,t)-\tilde m_{R}(x,t)|^\beta\right\rangle\le
C (R+1)^d e^{-2\beta\kappa t I\left(\frac{
      R}{2\kappa t}\right)} e^{H(\beta t)},
$$
where $I:=y\sinh^{-1}y-\sqrt{1+y^2}+1$.
\end{lemma}
We will also define for $\bf i\in\mathcal I$

\begin{equation}
\label{totalmass}
\tilde m_{\bf i}(t):=\sum_{x\in\Lambda''_{\bf i}}\tilde m_{\Lambda''_{\bf i}}(x,t).
\end{equation}
Let

\begin{equation}
\label{defese}
s(t):=e^{g(t)},
\end{equation}
where $g(t)$ is defined in (\ref{gdefi}).
Theorem \ref{theorem3.2} states that

$$
\lim_{t\to\infty}\frac{|\Lambda_L|}{s(t)}(m_L(t)-A(t))=S_\alpha,
$$
in distribution. By Lemma \ref{lemma-truncation},
the decomposition (\ref{mesodecomp}) and the choice
of scales (\ref{scalel}), (\ref{scalelm}) and (\ref{scaler}), it is enough to prove that

\begin{equation}
\label{enmi}
\lim_{t\to\infty}\frac{1}{s(t)}\left(
\sum_{{\bf i}\in\mathcal I} \tilde m_{\bf i}-|\Lambda_L|A(t)\right)=S_\alpha.
\end{equation}
\medskip
\subsubsection{Criteria for convergence to stable laws}
\label{criteria}
In order to prove part $(iii)$ of
Theorem \ref{theorem3.2}, we recall the following
result which gives conditions for a triangular array of 
independent random
variables to converge to a given infinite divisible distribution (see
for example Theorems  7 and 8 of Chapter 4 of Petrov \cite{petrov}
or \cite{bbm05}).

A random variable is infinite divisible if and only if
its characteristic function $\phi(t)$ admits the expression
(see for example Theorem 5 of Chapter 2 of Petrov \cite{petrov})

\begin{equation}
\label{cff}
\phi(t)=\exp\left\{i\nu t-\frac{\sigma^2t^2}{2}+
\int_{-\infty}^\infty\left(e^{itx}-1-\frac{itx}{1+x^2}\right)d\mathcal L(x)\right\},
\end{equation}
where $\nu$ is a real constant, $\sigma^2$ a non-negative constant,
the function $\mathcal L(x)$ is non-decreasing in $(-\infty,0)$ and $(0,\infty)$
and satisfies $\lim_{x\to\infty}\mathcal L(x)=\lim_{x\to-\infty}\mathcal L(x)=0$ 
and
for every $\delta>0$, $\int_{x:0<|x|\le\delta} x^2d\mathcal L(x)<
\infty$. $\mathcal L$ is called the {\it L\'evy-Khintchine spectral function}.
We will denote by $X_{\nu,\sigma,\mathcal L}$ the infinite divisible random
variable with characteristic function (\ref{cff}).

\medskip

\begin{theorem}
\label{theorem-criteria}
For each $t\ge 0$ let $\mathcal S(t)$ be a growing set
of indexes and consider a set $\{Y_{\bf i}(t):{\bf i}\in\mathcal S(t)\}$ 
 of independent
i.i.d. random variables. Call $P_t$  the
law common law of these random variables, say of $Y_{\bf 0}(t)$,
where we just call ${\bf 0}$ an arbitrary element of $\mathcal S(t)$. Assume that for every $\epsilon>0$ it is true that,

$$\lim_{t\to\infty} P_t( Y_{\bf 0}(t)>\epsilon)=0.$$
Now let ${\mathcal L}(x):{\mathbb R}/\{0\}\to {\mathbb R}$
be a L\'evy-Khintchine spectral function, $\nu\in{\mathbb R}$ and
$\sigma>0$, and let $A(t):[0,\infty)\to\mathbb R$ be some function. 
Then, if $n(t):=|\mathcal S(t)|$ the following statements are equivalent,

\begin{itemize}
\item[(i)] We have that

$$\lim_{t\to\infty}\left(\sum_{{\bf i}\in\mathcal S(t)} Y_{\bf i}(t)-A(t)\right)=X_{\nu,\sigma,{\mathcal L}},$$
where the convergence is in distribution.

\item[(ii)] Define for $y>0$ the truncated random variable at level $y$ as
$Z_y(t):=Y_{\bf 0}(t){\mathbf 1}_{|Y_1(t)|\le y}$. Also, let
$E_t(\cdot )$ and $Var_t(\cdot )$ denote the expectation and variance
corresponding to the law $P_t$. Then if $x$ is a continuity point of
${\mathcal L}$,

\begin{equation}
\label{od3}
{\mathcal L}(x)=\left\{
\begin{aligned}
&\lim_{t\to\infty}n(t)P_t( Y_{\bf 0}(t)\le x)&\qquad&\text{for}\ \ 
x<0,\\
-&\lim_{t\to\infty}n(t)P_t(Y_{\bf 0}(t) > x)&\qquad&\text{for}\ \ x>0,
\end{aligned}\right. 
\end{equation}

\begin{equation}
\label{od3.5}
\sigma^2=\lim_{y\to 0}\lim_{t\to\infty}n(t) Var_t (Z_y(t)),
\end{equation}
and for any $y>0$ which is a continuity point of ${\mathcal L}(x)$,

\begin{equation}
\label{od4}
\nu=\lim_{n\to\infty}
 \left(n(t)E_t(Z_y(t))-A(t)\right)+\int_{|x|>y}\frac{x}{1+x^2}d{\mathcal L}(x)
-\int_{y\ge |x|>0}\frac{x^3}{1+x^2}d{\mathcal L}(x).
\end{equation}

\end{itemize}
\end{theorem}

\medskip

\noindent We will apply Theorem \ref{theorem-criteria},
 to the set of i.i.d. random
variables $\{Y_{\bf i}:{\bf i}\in\mathcal I\}$ with

\begin{equation}
\label{yi}
Y_{\bf i}:=\frac{1}{s(t)}\tilde m_{\bf i},
\end{equation}

\begin{equation}
\label{definition-n}
n(t):=\left(\frac{L(t)}{l(t)}\right)^d=
\frac{e^{\gamma\frac{\tau(t)^{\rho'}}{\rho'}}}{e^{dt}}.
\end{equation}
and

\begin{equation}
\nonumber
\tilde A(t):=\begin{cases}
0\quad &{\rm if}\quad 0<\gamma<\gamma_1\\
E_\mu[Y_{\bf 0}, Y_{\bf 0}\le 1]\quad
&{\rm if}\quad \gamma=\gamma_1\\
E_\mu[Y_{\bf 0}]\quad &{\rm if}\quad \gamma_1<\gamma<\gamma_2
\end{cases}
\end{equation}
The first, second and third cases correspond to
the definition of $A(t)$ in (\ref{atdef}) of Theorem \ref{theorem3.2}.

\medskip
\subsubsection{L\'evy-Khintchine spectral function}
\label{clsf}
Here we will verify that the set of i.i.d. random variables
$\{Y_{\bf i}:{\bf i}\in\mathcal I\}$ defined in (\ref{yi})
satisfy condition (\ref{od3}) of part $(ii)$ of Theorem \ref{theorem-criteria}, with
the L\'evy-Khintchine spectral function

\begin{equation}
\label{levyalpha}
\mathcal L(x):=
\begin{cases}
0 &\qquad {\rm for}\ x\le 0\\
-\frac{1}{x^{\alpha(\gamma,\rho)}} &\qquad {\rm for}\ x>0,
\end{cases}
\end{equation}
where $\alpha(\gamma,\rho)$ is defined in (\ref{stable-exponent}).
To prove that condition (\ref{od3}) is satisfied with the
L\'evy-Khintchine spectral function (\ref{levyalpha}) it will
be enough to show that the following proposition is satisfied.

\medskip

\begin{proposition}
\label{propopo}
Consider $\tilde m_{\bf 0}$ defined
in (\ref{totalmass}), with $v$ an i.i.d. potential with Weibull law $\mu$.
Then, for   $s(t)$ defined  in (\ref{defese}) and $n(t)$ in
(\ref{definition-n}) we have
that for all $u>0$ it is true that

$$
\lim_{t\to\infty} n(t)\mu\left(\frac{\tilde m_{\bf 0}}{s(t)}
>u\right)
=\frac{1}{u^{\alpha(\gamma,\rho)}}.
$$
\end{proposition}

\begin{proof} 
 We will  use the expansion

\begin{equation}
\label{spectral-expansion}
\tilde m_{\Lambda_0''}(x,t)=\sum_{n=0}^N
e^{t\lambda_n}\psi_n(x)(\psi_n,1(\Lambda''_{\bf 0})),
\end{equation}
where $\{\lambda_n:0\le n\le N\}$ and
$\{\psi_n:0\le n\le N\}$ are the eigenvalues in
decreasing order and
eigenfunctions, respectively, of the Schr\"odinger operator
$H_{\Lambda_{\bf 0}'',v}$  on $\Lambda_0''$
with Dirichlet boundary conditions.
Furthermore, we will need
to choose $\delta_1$ and $\delta_2$ so   that $0<\delta_1<\delta_2$
and

\begin{equation}
\label{delta23}
2(1-\delta_2)^{\rho}>1
\end{equation}
and to consider the events

$$
A_1:=\left\{v_{(1)}\ge (1-\delta_1)\mathcal A_0
t^{\rho'-1}\right\},
$$
and

$$
A_2:=\left\{ v_{(2)}\le (1- \delta_2)\mathcal A_0
 t^{\rho'-1}\right\}.
$$

\medskip

\noindent {\it Step 1.} 
Note that

\begin{eqnarray}
\nonumber
&  \mu\left(\frac{\tilde m_{\bf 0}}{s(t)}
>u\right)
\nonumber
= \mu\left(A_1,A_2,\frac{\tilde m_{\bf 0}}{s(t)}
>u\right)\\
\label{lasta1a2}&+
\mu\left(A_1^c,A_2,\frac{\tilde m_{\bf 0}}{s(t)}
>u\right)
+
\mu\left(A_1,A_2^c,\frac{\tilde m_{\bf 0}}{s(t)}
>u\right).
\end{eqnarray}
By inequality (\ref{boundv2}) of part $(ii)$ of Lemma \ref{lemma1}, we have that

\begin{equation}
\label{mulimg}
\mu(A_2^c)=\mu\left(v_{(2)}>(1-\delta_2)
\left(\frac{\gamma\rho}{\rho'}\right)^{1/\rho}
 t^{\rho'-1}\right)\le 
e^{-2(1-\delta_2)^\rho \frac{\gamma}{\rho'} t^{\rho'}}.
\end{equation}
But from (\ref{delta23}) we have that

\begin{equation}
\label{ofstep}
\lim_{t\to\infty} n(t) \mu(A_2^c)=0.
\end{equation}
It follows  from (\ref{lasta1a2}) and from (\ref{ofstep}) that

\begin{eqnarray}
\nonumber
\displaystyle &\lim_{t\to\infty}n(t)
\mu\left(\frac{\tilde m_{\bf 0}}{s(t)}
>u\right)
=
\lim_{t\to\infty}n(t)\mu\left(A_1,A_2,\frac{\tilde m_{\bf 0}}{s(t)}
>u\right)\\
\nonumber
&+
\lim_{t\to\infty}n(t)\mu\left(A_1^c,A_2,\frac{\tilde m_{\bf 0}}{s(t)}
>u\right).
\end{eqnarray}

\medskip

\noindent {\it Step 2.}
From the expansion (\ref{spectral-expansion}) 
 note that on the event $A_1^c$ we have that

\begin{eqnarray}
\nonumber
&\tilde m_{\bf 0}=
e^{t\lambda_0}\sum_{n=0}^N e^{-t(\lambda_0-\lambda_n)}
(\psi_n,1(\Lambda''_{\bf 0}))^2\\
\nonumber
&\le 
e^{t\lambda_0}l^{2d}\le e^{t(\lambda_0+2d)}\le
e^{t(v_{(1)}+2d)}\le e^{(1-\delta_1)\mathcal A_0
 t^{\rho'}+2d t},
\end{eqnarray}
where in the second to last inequality we have used 
(\ref{eu5}). Using the fact that $\lim_{t\to\infty}e^{-g(t)}e^{(1-\delta_1)\mathcal A_0 t^{\rho'}+2d t}=0$,
we therefore see that $\lim_{t\to\infty}\tilde m_{\bf 0} e^{-g(t)}=0$ and hence

$$
\lim_{t\to\infty} n(t) \mu(A_1^c,\tilde m_{\bf 0} 
e^{-g(t)}>u)=0.
$$
It follows from  (\ref{lasta1a2}) that

\begin{equation}
\label{step3}
\lim_{t\to\infty}n(t)
\mu\left(\frac{\tilde m_{\bf 0}}{s(t)}
>u\right)
=
\lim_{t\to\infty}n(t)\mu\left(A_1,A_2, \frac{\tilde m_{\bf 0}}{s(t)}
>u\right).
\end{equation}

\medskip

\noindent {\it Step 3.}
By part $(iii)$ of Theorem \ref{eigenvf}, on the event $A_1\cap A_2$,
the normalized principal Dirichlet eigenfunction $\psi_0$ is such that
for some $x_0\in\Lambda''_{\bf 0}$ one has that

\begin{equation}
\label{psi0}
\psi_0(x)=1+\varepsilon(x),
\end{equation}
where

$$
|\varepsilon(x)|\le \frac{C_4}{t^{(|x-x_0|_1+1)(\rho'-1)}}.
$$
for some constant $C_4$.
Therefore, from (\ref{psi0}) and the expansion (\ref{spectral-expansion}) 
we can see that on the event $A_1\cap A_2$,

$$
\tilde m_{\bf 0}=
e^{t\lambda_0}\left(1+O\left(\frac{1}{t^{\rho'-1}}\right)\right)+\sum_{n=1}^N
e^{t\lambda_n}(\psi_n,1(\Lambda''_{\bf 0}))^2.
$$
Therefore, from the identity (\ref{lambdaexps}) of part $(i)$ of Theorem \ref{eigenvf}
and part $(iv)$ of the same theorem,
we see that on $A_1\cap A_2$ it is true that

\begin{eqnarray}
\nonumber
& e^{t\lambda_0}\left(1+O\left(\frac{1}{t^{\rho'-1}}\right)\right)\le
\tilde m_{\bf 0}\\
\nonumber
&\le
e^{t\lambda_0}\left(1+O\left(\frac{1}{t^{\rho'-1}}\right)+(2l+1)^2
e^{-(\delta_2-\delta_1)\left(\frac{\gamma\rho}{\rho'}\right)^{1/\rho} t^{\rho'}}\right)\\
\nonumber
&=
e^{t\lambda_0}
\left(1+O\left(\frac{1}{t^{\rho'-1}}\right)\right).
\end{eqnarray}
Hence, on $A_1\cap A_2$,

\begin{equation}
\label{inequalities}
\tilde m_{\bf 0}=
e^{t \lambda_0}
\left(1+O\left(\frac{1}{t^{\rho'-1}}\right)\right).
\end{equation}
\medskip

\noindent {\it Step 4.}
 Using 
 (\ref{inequalities}) in display (\ref{step3}) we see that
for any sequence $\{t_k:n\ge 1\}$ such that
$\lim_{k\to\infty}n(t_k)\mu\left(\frac{\tilde m_{\bf 0}(t_k)}{s(t_k)}\right)$
exists (possibly being equal to $\infty$), one has that
\begin{eqnarray}
\nonumber
&\lim\limits_{k\to\infty}n(t_k)
\mu\left( \frac{\tilde m_{\bf 0}(t_k)}{s(t_k)}
>u\right)\\
\nonumber
&=
\lim\limits_{k\to\infty} n(t_k)
\mu\left(A_1,A_2,e^{t_k\lambda_0}\left(1+O\left(\frac{1}{t_k^{\rho'-1}}\right)\right)\ge u 
e^{g(t_k)}\right)\\
\label{88}
&
=\lim\limits_{k\to\infty} n(t_k)\mu\left(e^{t_k\lambda_0}
\left(1+O\left(\frac{1}{t_k^{\rho'-1}}\right)\right)\ge u
e^{g(t_k)}\right),
\end{eqnarray}
where in the last step we used (\ref{ofstep}) of  Step 2 and the fact
that since on $A_1^c$ it is true that 
$\lambda_0\le (1-\delta_1)g(t)$, eventually in $t$,

\begin{equation}
\label{lambdaa1}
e^{t\lambda_0}\le x e^{(1-\delta_1)g(t)}.
\end{equation}
 From (\ref{88}) we now get that
there is a constant $C_1$ such that

\begin{eqnarray}
\nonumber
&
\liminf\limits_{t\to\infty} n(t) 
\mu\left(\lambda_0\ge h(t) +\frac{1}{t}\log u
+\frac{C_1}{t^{\rho'}}
\right)
\le\liminf\limits_{t\to\infty}n(t)
\mu\left( \frac{\tilde m_{\bf 0}}{s(t)}
>u\right)
\\
\label{89}
&\le
\limsup\limits_{t\to\infty}n(t)
\mu\left( \frac{\tilde m_{\bf 0}}{s(t)}
>u\right)
\le
\limsup\limits_{t\to\infty} n(t) 
\mu\left(\lambda_0\ge h(t) +\frac{1}{t}\log u
-\frac{C_1}{t^{\rho'}}
\right).
\end{eqnarray}

\medskip
\noindent {\it Step 5.}
For each $t>0$, define $W_t$ as the set of potentials
$v$ such that $v(y)=0$ for $y\notin \Lambda_{\bf 0}''$.
For each $v\in W_t$ and $x\in \Lambda_{\bf 0}''$,
consider  the function

\begin{equation}
\nonumber
B(s,v,x):=
\frac{s+2d\kappa}{1+\sum_{j=1}^\infty
\sum_{\gamma\in{\mathbb P}^{2j+1}(x,x)}
\prod_{z\in\gamma,z\ne z_1} \frac{\kappa}{s+2d\kappa-v_0(z)}},
\end{equation}
which is well defined whenever $s>\bar v:=\max\{v(x):x\in\Lambda_{\bf 0}''\}$.
Using the fact that the function $B(s,v)$ is increasing
on $[\bar v,\infty)$
and part $(i)$ of Theorem \ref{eigenvf}, note that
on the event $A_1\cap A_2$,
the inequality
$\lambda_0\ge s$ is equivalent to $v_{(1)}\ge B(s,v,x_0)$,
where $x_0\in\Lambda_{\bf 0}''$ is such that $v_{(1)}=v(x_0)$.
It follows  from (\ref{ofstep}) of Step 2 and (\ref{lambdaa1}) of Step 4  that

\begin{eqnarray}
\nonumber
&\liminf\limits_{t\to\infty}n(t)\mu\left(\lambda_0\ge h(t) +\frac{1}{t}\log u
+\frac{C_1}{t^{\rho'}}\right)\\
\nonumber
&=
\liminf\limits_{t\to\infty} n(t)
\mu\left(A_1,A_2,\lambda_0\ge h(t) +\frac{1}{t}\log u
+\frac{C_1}{t^{\rho'}}\right)\\
\label{rhsa1}
&=\liminf\limits_{t\to\infty} n(t)
\mu\left(A_1,A_2,v_{(1)} \ge B\left( h(t) +\frac{1}{t}\log u
+\frac{C_1}{t^{\rho'}},v,x_0\right)\right).
\end{eqnarray}
Now, note that

$$
A_1\cap A_2=\bigcupdot_{x\in\Lambda_{\bf 0}''}\left\{v(x)\ge
(1-\delta_1)\mathcal A_0 t^{\rho'-1}\right\}\cap A_2,
$$
where the symbol $\dotcup$ denotes a disjoint union.
Therefore, the probability appearing in the right-most hand side
of display (\ref{rhsa1}) can be written as

\begin{eqnarray}
\nonumber
&\mu\left(A_1,A_2,v_{(1)} \ge B\left( h(t) +\frac{1}{t}\log u
+\frac{C_1}{t^{\rho'}},v,x_0\right)\right)\\
\label{muut}
&\!=\! \! \sum\limits_{x\in\Lambda_{\bf 0}''}\! \mu\!\left(A_2,
v(x)\ge
(1-\delta_1)\mathcal A_0 t^{\rho'-1},
v(x)\ge B\left( h(t) +\frac{1}{t}\log u
+\frac{C_1}{t^{\rho'}},v,x\right)\right).
\end{eqnarray}
Furthermore, for each $x\in\Lambda_{\bf 0}''$, on the event 
$A_2\cap\{v(x)\ge (1-\delta_1)\mathcal A_0 t^{\rho'-1}\}$
whenever $s\ge (1-\delta_1)t^{\rho'-1}$,
 we have that

$$
\sum_{j=1}^\infty
\sum_{\gamma\in{\mathbb P}^{2j+1}(x,x)}
\prod_{z\in\gamma,z\ne z_1} \frac{\kappa}{s+2d\kappa-v_0(z)}=O\left(
\frac{1}{t^{\rho'-1}}\right),
$$
Now, by part $(ii)$ of Theorem \ref{theorem3.1}, we have that

\begin{equation}
\label{gestim}
 h(t)=h_0(t)
 +O\left(\frac{1}{t^{\rho'-1}}\right).
\end{equation}
Hence, $h(t)=\mathcal A_0\tau^{\rho'-1}+O(1)$ and
on the event $A_2\cap\{v(x)\ge (1-\delta_1)\mathcal A_0 t^{\rho'-1}\}$
we have

\begin{eqnarray}
\nonumber
&B\left( h(t) +\frac{1}{t}\log u
+\frac{C_1}{t^{\rho'}},v,x\right)
=h(t)+\frac{1}{t}\log u+\frac{C_1}{t^{\rho'}}+2d\kappa
+O\left(
\frac{h(t)}{t^{\rho'-1}}\right)\\
\label{fin10}
&\ge \mathcal A_0\tau^{\rho'-1}+
O\left(1\right)
\end{eqnarray}
and also that

\begin{eqnarray}
\nonumber
&B\left( h(t) +\frac{1}{t}\log u
+\frac{C_1}{t^{\rho'}},v,x\right)=
B_M\left( h(t) +\frac{1}{t}\log u,\theta_x v\right)
+
O\left(\frac{1}{t^{\rho'}}\right)
+
O\left(
\frac{1}{t^{2M(\rho'-1)}}\right)\\
\label{fin11}
&=B_M\left( h(t) 
,\theta_x v\right)
+
\frac{1}{t}\log u+
O\left(\frac{1}{t^{\rho'}}\right)
+
O\left(
\frac{1}{t^{2M(\rho'-1)}}\right).
\end{eqnarray}
where $\{\theta_x:x\in\mathbb Z^d\}$ is the canonical
set of translations acting on the potentials $v\in W_t$ as
$\theta_xv(y):=v(x+y)$.
It follows from (\ref{fin10}) that eventually in $t$ for all $x\in\Lambda_{\bf 0}''$
one has the inclusion

\begin{equation}
\label{fin12}
\{v(x)\ge
(1-\delta_1)\mathcal A_0 t^{\rho'-1}\}\subset
\left\{
v(x)\ge B\left( h(t) +\frac{1}{t}\log u
+\frac{C_1}{t^{\rho'}},\theta_x v\right)\right\}.
\end{equation}
Hence, going back to (\ref{muut}) we see after considering
(\ref{fin11}) and (\ref{fin12}) that
eventually in $t$ it is true that

\begin{eqnarray}
\nonumber
&
\mu\left(A_1,A_2,v_{(1)} \ge B\left( h(t) +\frac{1}{t}\log u
+\frac{C_1}{t^{\rho'}},v,x_0\right)\right)\\
\label{mua1a2}
&\!\!\!\!\!\!\!\!\!=\!\!\!
\sum\limits_{x\in\Lambda_{\bf 0}''}\mu\left(A_2,
v(x)\ge B_M\left( h(t),\theta_x v\right)
+\frac{1}{t}\log u
+O\left(\frac{1}{t^{\rho'}}\right)
+
O\left(
\frac{1}{t^{2M(\rho'-1)}}\right)
\right).
\end{eqnarray}
From the bound (\ref{mulimg}) of Step 3, using the equality (\ref{mua1a2}) we conclude that
in fact there is a constant $C_2$ such that
\begin{eqnarray}
\nonumber
&
\liminf\limits_{t\to\infty}n(t)\mu\left(A_1,A_2,v_{(1)} \ge B\left( h(t) 
+\frac{1}{t}\log u
+\frac{C_1}{t^{\rho'}},v,x_0\right)\right)\\
\label{summs}
&\ge
\liminf\limits_{t\to\infty}n(t)\sum_{x\in\Lambda_{\bf 0}''}\mu\left(
v(x)\ge B_M\left( h(t),\theta_x v\right)
+\frac{1}{t}\log u+
\frac{C_2}{t^{\rho'}}
+\frac{C_2}{t^{2M(\rho'-1)}}\right)\\
\nonumber
&=
\liminf\limits_{t\to\infty}n(t)|\Lambda_{\bf 0}''|
\mu\left(v(0)\ge B_M\left( h(t), v\right)
+\frac{1}{t}\log u+
\frac{C_2}{t^{\rho'}}
+\frac{C_2}{t^{2M(\rho'-1)}}\right),
\end{eqnarray}
where in the last equality we have used the fact that
the terms under the summation in (\ref{summs}) are
translation invariant when $x\in\Lambda_{\bf 0}$ is far enough
of the boundary, and that the points which do not have this property
have a negligible cardinality with respect to $|\Lambda_{\bf 0}''|$.
Combining this with (\ref{rhsa1}) we get that

\begin{eqnarray}
\nonumber
&
\liminf\limits_{t\to\infty}n(t)\mu\left(\lambda_0\ge h(t) +\frac{1}{t}\log u
+\frac{C_1}{t^{\rho'}}\right)\\
\label{combine}
&\ge
\liminf\limits_{t\to\infty} e^{-\gamma\frac{\tau^{\rho'}}{\rho'}}
E_\mu\left[e^{-\frac{1}{\rho}\left(B_M\left( h(t), v\right)
+\frac{1}{t}\log u+
\frac{C_2}{t^{\rho'}}
+\frac{C_2}{t^{2M(\rho'-1)}}\right)^\rho}\right].
\end{eqnarray}
Now, for $x>1$ and $\epsilon<1$ one has that

$$
(x+\epsilon)^\rho=x^\rho+\rho x^{\rho-1}\epsilon+
O\left(\frac{\epsilon^2}{x^{2-\rho}}\right).
$$
Calling for the moment $B_M=B_M(h(t),v)$, we see from
the lower bound (\ref{gestim}) of Step 1 that this implies that

\begin{eqnarray}
\nonumber
&\frac{1}{\rho}
\left(B_M+\frac{1}{t}\log u +O\left(\frac{1}{t^{\rho'}}\right)+
O\left(\frac{1}{t^{2M(\rho'-1)}}\right)
\right)^\rho\\
\label{bemet}
&=
\frac{1}{\rho}B_M^\rho+ \frac{B_M^{\rho-1}}{t}\log u +
+
O\left(\frac{B_M^{\rho-1}}{t^{2M(\rho'-1)}}\right)
+
\frac{1}{B_M^{2-\rho}}\left(
O\left(\frac{1}{t^2}\right)
+
O\left(\frac{1}{t^{4M(\rho'-1)}}\right)\right).
\end{eqnarray}
Now, by the bounds (\ref{gestim}) and (\ref{bemet}), and the fact that
$B_M\ge h(t)+2d\kappa$,
on the event $A_2$ we have that

\begin{equation}
\label{atr}
\mathcal A_0 \tau^{\rho'-1}\le B_M\le \mathcal A_0 \tau^{\rho'-1}+\frac{C_5}{t^{\rho'-1}}.
\end{equation}
Combining (\ref{atr}) with (\ref{bemet}) we conclude that
on the event $A_2$ one has that

\begin{eqnarray}
\nonumber
&
\frac{1}{\rho}
\left(B_M+\frac{1}{t}\log u +O\left(\frac{1}{t^{\rho'}}\right)+
O\left(\frac{1}{t^{2M(\rho'-1)}}\right)
\right)^\rho\\
\nonumber
&=
\frac{1}{\rho}B_M^\rho+\alpha \log u +
O\left(\frac{1}{t^{\rho'}}\right)+
O\left(\frac{1}{t^{2M\rho'-(2M+1)}}\right).
\end{eqnarray}
Inserting this in (\ref{combine}) we conclude that

\begin{eqnarray}
\nonumber
&
\liminf\limits_{t\to\infty}n(t)\mu\left(\lambda_0\ge h(t) +\frac{1}{t}\log u
+\frac{C_1}{t^{\rho'}}\right)\\
\label{lambdalower}
&\ge
\liminf\limits_{t\to\infty} e^{-\gamma\frac{\tau^{\rho'}}{\rho'}}
\frac{1}{u^{\alpha}}E_\mu\left[e^{-\frac{1}{\rho}B_M\left( h(t), v\right)+
\frac{C_5}{t^{\rho'}}
+\frac{C_5}{t^{2M\rho'-(2M+1)}}}
\right]=\frac{1}{u^\alpha},
\end{eqnarray}
where in the last equality we have used the definition of
$M$ given in (\ref{mdef}), which implies that

$$
2M\rho'-(2M+1)>1,
$$ 
and the definition of $B_N$ 
 given in (\ref{bnsv}).
By a similar argument, where in order to control the terms
close to the boundary we have to use the fact that
for all $x\in\Lambda_{\bf 0}''$ it is true that
for some constant $C_6$
\begin{eqnarray}
\nonumber
&
\mu\left(
v(x)\ge B_M\left( h(t) +\frac{1}{t}\log u,\theta_x v\right)-
\frac{C_6}{t^{\rho'}}
-\frac{C_6}{t^{2M\rho'-(2M+1)}}
\right)\\
\nonumber
&\le
\mu\left(
v(0)\ge B_M\left( h(t) +\frac{1}{t}\log u, v\right)-
\frac{C_6}{t^{\rho'}}
-\frac{C_6}{t^{2M\rho'-(2M+1)}}
\right),
\end{eqnarray}
we can get that

\begin{equation}
\label{lambdaupper}
\limsup\limits_{t\to\infty}n(t)\mu\left(\lambda_0\ge h(t) +\frac{1}{t}\log u
-\frac{C_1}{t^{\rho'}}\right)\le\frac{1}{u^\alpha}.
\end{equation}
Inserting (\ref{lambdalower}) (\ref{lambdaupper}) in (\ref{89})
we finish the proof of the proposition.

\end{proof}

\medskip

\subsubsection{The truncated moments}
\label{csigma}
Here we will compute some quantities related to the moments
of the random variable $Y_{\bf 0}$ defined in (\ref{yi})
which will be later used to prove that conditions  (\ref{od3.5})
and (\ref{od4}) are satisfied.

\medskip

\begin{lemma}
\label{lsigma} Consider
 the random variable $Y_{\bf 0}$
given in (\ref{yi}).
Assume that $0<\gamma<\gamma_2$. Then the following statements
are satisfied.

\begin{itemize}

\item[(i)] For $y>0$ we have that

\begin{equation}
\label{limity}
\lim_{t\to\infty}n(t) \left(E_\mu\left[Y_{\bf 0}(t),Y_{\bf 0}(t)\le y\right]
-\tilde A(t)\right)=
\begin{cases}
 \frac{\alpha}{1-\alpha}y^{1-\alpha}&\quad{\rm if} \gamma\in (0,\gamma_1)\cup
(\gamma_1,\gamma_2)\\
\log y &\quad{\rm if}\ \gamma=\gamma_1.
\end{cases}
\end{equation}

\item[(ii)] For $y>0$ we have that

\begin{equation}
\nonumber
\lim_{t\to\infty}n(t) E_\mu\left[Y^2_{\bf 0}(t),Y_{\bf 0}(t)\le y\right]=
 \frac{\alpha}{2-\alpha}y^{2-\alpha}.
\end{equation}
\end{itemize}
Here $\alpha=\alpha(\gamma,\rho)$ is defined in (\ref{stable-exponent}).
\end{lemma}
\begin{proof} {\it Part (i).}
We will first prove  (\ref{limity}) for the case $0<\gamma<\gamma_1$.
Note that for each $N$ we have that

\begin{equation}
\label{simil3}
E_\mu\left[Y_{\bf 0}(t)
 1(Y_{\bf 0}(t)
\le y)\right]=
\sum_{i=0}^{N-1}E_\mu\left[
Y_{\bf 0}(t)
 1\left(\frac{i}{N}y\le Y_{\bf 0}(t)\le \frac{i+1}{N}y \right)\right].
\end{equation}
Now, for each $1\le i\le N-1$ we have by Proposition \ref{propopo} that

\begin{eqnarray}
\nonumber
&\varlimsup_{t\to\infty} n(t)
E_\mu\left[Y_{\bf 0}(t)
 1\left(\frac{i}{N}y\le Y_{\bf 0}(t)
\le \frac{i+1}{N}y \right)\right]\\
\label{simil2}
&\le
\frac{i+1}{N}y
\varlimsup_{t\to\infty} n(t)
 \mu\left( \frac{i}{N}y\le Y_{\bf 0}(t)
\le  \frac{i+1}{N}y
\right)
=
\frac{i+1}{N}y \left(\frac{1}{(iy/N)^\alpha}-
\frac{1}{((i+1)y/N)^\alpha}\right).
\end{eqnarray}
Similarly, we can conclude that

\begin{equation}
\label{simil}
\varliminf_{t\to\infty} n(t)E_\mu\left[Y_{\bf 0}(t)
1\left(\frac{i}{N}y\le Y_{\bf 0}(t)
\le \frac{i+1}{N}y \right)\right]\ge 
\frac{i}{N}y \left(\frac{1}{(iy/N)^\alpha}-
\frac{1}{((i+1)y/N)^\alpha}\right).
\end{equation}
Combining (\ref{simil2}) and (\ref{simil}) with (\ref{simil3}) we get

\begin{eqnarray}
\nonumber
&\sum_{i=0}^{N-1}
\frac{i}{N}y \left(\frac{1}{(iy/N)^\alpha}-
\frac{1}{((i+1)y/N)^\alpha}\right)\le
\varliminf_{t\to\infty}n(t)E_\mu\left[Y_{\bf 0}(t)
 1(Y_{\bf 0}(t)
\le y )\right]\\
\label{simil4}
&\le 
\varlimsup_{t\to\infty}n(t)E_\mu\left[Y_{\bf 0}(t)
 1(Y_{\bf 0}(t)
\le y )\right]\le
\sum_{i=0}^{N-1}\frac{i+1}{N}y \left(\frac{1}{(iy/N)^\alpha}-
\frac{1}{((i+1)y/N)^\alpha}\right).
\end{eqnarray}
Now note that for $0<\gamma<\gamma_1$, one has that
$0<\alpha<1$. We can hence take the limit when $N\to\infty$ in (\ref{simil4}) 
to deduce that

$$
\lim_{t\to\infty}n(t)E_\mu\left[Y_{\bf 0}(t)
 1(Y_{\bf 0}(t)
\le y )\right]
=
\int_0^y\frac{1}{x^\alpha}dx-y^{1-\alpha}=\frac{\alpha}{1-\alpha}y^{1-\alpha},
$$
which proves (\ref{limity}) for the case $0<\gamma<\gamma_1$. The proof
of (\ref{limity}) for the case $\gamma_1<\gamma<\gamma_2$ has
to take into account that

$$
E_\mu[Y_{\bf 0}(t),Y_{\bf 0}(t)\le y]-\tilde A(t)=-E_\mu[Y_{\bf 0}(t),Y_{\bf 0}(t)> y],
$$
and then follows an analysis similar to the previous case.
The proof of (\ref{limity}) in the case $\gamma=\gamma_1$ uses
the fact that

$$
E_\mu[Y_{\bf 0}(t),Y_{\bf 0}(t)\le y]-\tilde A(t)=
\begin{cases}
E_\mu[Y_{\bf 0}(t),1<Y_{\bf 0}(t)\le y]\quad{\rm if}&\quad y\ge 1\\
-E_\mu[Y_{\bf 0}(t),y<Y_{\bf 0}(t)<1]\quad{\rm if}&\quad y< 1,
\end{cases}
$$
which then enables one to prove that

$$
\lim_{N\to\infty}\left(E_\mu[Y_{\bf 0}(t),Y_{\bf 0}(t)\le y]-\tilde A(t)\right)=\int_1^y
\frac{1}{x}dx=\log y.
$$

\smallskip

\noindent {\it Part (ii).}
In analogy to the inequalities (\ref{simil4}), we
can conclude that

\begin{eqnarray}
\nonumber
&\sum_{i=0}^{N-1}
\left(\frac{i}{N}y\right)^2 \left(\frac{1}{(iy/N)^\alpha}-
\frac{1}{((i+1)y/N)^\alpha}\right)\le
\varliminf_{t\to\infty}n(t)E_\mu\left[Y^2_{\bf 0}(t)
 1(Y_{\bf 0}(t)
\le y )\right]\\
\nonumber
&\le 
\varlimsup_{t\to\infty}n(t)E_\mu\left[Y^2_{\bf 0}(t)
 1(Y_{\bf 0}(t)
\le y )\right]\le
\sum_{i=0}^{N-1}
\left(\frac{i+1}{N}y\right)^2 \left(\frac{1}{(iy/N)^\alpha}-
\frac{1}{((i+1)y/N)^\alpha}\right).
\end{eqnarray}
As in (\ref{simil4}), since $0<\gamma<\gamma_2$ implies that
$0<\alpha<2$, we can take the limit as $N\to\infty$ above
to deduce that

$$
\lim_{t\to\infty}n(t)E_\mu\left[Y^2_{\bf 0}(t)
 1(Y_{\bf 0}(t)
\le y )\right]
=
\int_0^y\frac{1}{x^{\alpha-1}}dx-y^{2-\alpha}=\frac{\alpha}{2-\alpha}y^{2-\alpha},
$$
\end{proof}

\medskip

\subsubsection{The parameters of the infinite divisible law}
\label{cbeta}
Here we will show that the set of i.i.d. random variables
$\{Y_{\bf i}:{\bf i}\in\mathcal I\}$ satisfy conditions (\ref{od3.5})
and (\ref{od4})
of part $(ii)$ of Theorem \ref{theorem-criteria} with

\begin{equation}
\label{sigmaeq}
\sigma^2=\lim_{y\to 0}\lim_{t\to\infty}n(t)Var_t (Z_y(t))=0,
\end{equation}
and

\begin{equation}
\label{betaformula}
\nu:=
\begin{cases}
\frac{\alpha\pi}{2\cos\frac{\alpha\pi}{2}}&\quad{\rm if}\quad \gamma\in(0,
\gamma_1)\cup (\gamma_1,\gamma_2)\\
0&\quad {\rm if}\quad \gamma=\gamma_1.
\end{cases}
\end{equation}
The proof of (\ref{sigmaeq}) is a direct consequence of
part $(ii)$ of Lemma \ref{lsigma}.
Since the proof of  (\ref{betaformula}) is completely analogous
to the proofs of Propositions 6.4 and 6.5 of \cite{bbm05},
we will just give an outline here. Note that by (\ref{od4}), part $(i)$
of Lemma \ref{lsigma} and (\ref{levyalpha}), we should have

$$
\nu=
\frac{\alpha}{1-\alpha}y^{1-\alpha}+\alpha\int_y^\infty 
\frac{x^{-\alpha}}{1+x^2}dx-\alpha \int_0^y \frac{x^{2-\alpha}}{1+x^2}dx.
$$
Using the identity $\int_0^\infty \frac{x^{-\alpha}}{1+x^2}dx=\frac{\pi}{2
\cos\frac{\alpha\pi}{2}}$ in the case $\gamma\ne\gamma_1$, we can obtain (\ref{betaformula}) for that case. A similar analysis gives the case 
$\gamma=\gamma_1$.

\medskip

\subsubsection{Conclusion} Gathering (\ref{levyalpha}), (\ref{sigmaeq}) and (\ref{betaformula}),
into Theorem \ref{theorem-criteria},
and the observation that it is enough to find the limiting
law in (\ref{enmi}), we conclude that for $0<\gamma<\gamma_2$,

$$
\lim_{t\to\infty}\frac{1}{s(t)}m_L(t)=Z,
$$
where the convergence is in distribution and $Z$ is an infinite divisible distribution with
characteristic function

\begin{equation}
\label{charst}
\phi(t)=\exp\left\{i\nu u+\alpha\int_0^\infty\left(e^{itx}-1-\frac{itx}{1+x^2}\right)
\frac{dx}{x^{\alpha+1}}\right\},
\end{equation}
with $\nu$ defined in (\ref{betaformula}).
Now, by Theorem 6.6 of \cite{bbm05}, we know that any infinite divisible
distribution corresponding to (\ref{charst}) has the canonical representation
given in (\ref{charstable}). This finishes the proof of  Theorem \ref{theorem3.2}.

\medskip

\section{Annealed asymptotics}
In this section we will prove
prove (\ref{annealed-asymp}) of Theorem \ref{theorem4}.

\smallskip

\begin{proposition}
\label{corollary1} Consider the solution $\{m(x,t):x\in{\mathbb Z}^d,
t\ge 0\}$ of the parabolic
Anderson equation (\ref{i1}) with Weibull potential of parameter
$\rho>1$. Then

$$
\langle m(0,t)\rangle\sim \left(\frac{\pi}{\rho-1}\right)^{1/2}
 t^{1-\frac{\rho'}{2}} e^{\frac{t^{\rho'}}{\rho'}-2d(\kappa t-t^{2-\rho'})}.
$$
\end{proposition}
\smallskip

\noindent  Define the
{\it cumulant generating function} by,

$$
H(t)=\log\left\langle e^{v(0)t}\right\rangle,\qquad t\ge 0.
$$
Using the independence of the coordinates of the potential $v$, note that,

$$
\left\langle m(0,t)\right\rangle =E_0\left[e^{\sum_{x\in{\mathbb Z}^d}
H({\mathcal L}(t,x))}\right],
$$
where $E_0$ is the expectation defined by the law $P_0$ of a
simple symmetric random walk on ${\mathbb Z}^d$, starting from
$0$, of total jump rate $1$ and for each $t\ge 0$ and
site $x\in{\mathbb Z}^d$, ${\mathcal L}(t,x)$ is the total time
spent by the random walk in the time interval $[0,t]$ at $x$.

The basis of our proof of Theorem \ref{theorem4} will be the
following result.
\smallskip

\begin{proposition}
\label{theorem1}
 Consider the solution $\{m(x,t):x\in{\mathbb Z}^d,
t\ge 0\}$ of the parabolic
Anderson equation (\ref{i1}) with Weibull potential of parameter
$\rho>1$. Then

\begin{equation}
\label{i1.1}
\langle m(0,t)\rangle\sim \left(\frac{\pi}{\rho-1}\right)^{1/2}
 e^{\frac{t^{\rho'}}{\rho'}-2d\kappa t}
\sum_{n=1}^\infty 
\frac{\kappa^{n-1}}{t^{(\rho'-1)n}}
\sum_{\gamma\in{\mathbb P}^n(0)} 
\sum_{i=1}^k  \frac{t^{\rho' \left(n_i-\frac{1}{2}\right)}}{ (n_i-1)!}.
\end{equation}
\end{proposition}
\smallskip

\smallskip
\subsubsection{Preliminary estimates}
To prove Proposition \ref{theorem1}, we will need the precise asymptotics
of the cumulant generating function, given by the following lemma.
\smallskip
\begin{lemma}
\label{cumulant}
 For $t\ge 0$,

\begin{equation}
\nonumber
H(t)=\frac{t^{\rho'}}{\rho'}+\frac{\rho'}{2}\log t+
\frac{1}{2}\log\frac{\pi}{\rho-1}+\varepsilon(t),
\end{equation}
where $\lim_{t\to\infty}\varepsilon(t)=0$.
\end{lemma}
\begin{proof} Note that

$$
\left\langle e^{tv(0)}\right\rangle =\int_0^\infty u^{\rho-1}
e^{-u^\rho/\rho} e^{ut}du.
$$
Making the variable change $u=w t^{1/(\rho -1)}$, this becomes,

\begin{equation}
\label{expression}
 t^{\rho'}
\int_0^\infty  w^{\rho-1}
e^{- t^{\rho'}\left(\frac{ w^\rho}{\rho}-w\right)}dw.
\end{equation}
The function $f(w):=\frac{w^\rho}{\rho}-w$, attains its
minimum value $-1/\alpha'$ at $w=1$, having the expansion,

\begin{equation}
\label{expansion}
f(w)=-\frac{1}{\rho'}+(w-1)^2(\rho-1)+(w-1)^3 f^{(3)}(\bar w),
\end{equation}
where $\bar w$ is between $1$ and $w$. Now, let $0<\epsilon<1$,
and make the decomposition,

\begin{equation}
\label{decomp}
\int_0^\infty  w^{\rho-1}
e^{- t^{\rho'}\left(\frac{ w^\rho}{\rho}-w\right)}dw
=A_1+A_2+A_3,
\end{equation}
where $A_1=\int_0^{1-\epsilon}  w^{\rho-1}e^{- t^{\rho'}
\left(\frac{ w^\rho}{\rho}-w\right)}dw$,
$A_2=\int_{1-\epsilon}^{1+\epsilon}  w^{\rho-1}
e^{- t^{\rho'}\left(\frac{ w^\rho}{\rho}-w\right)}dw$
and $A_3=\int_{1+\epsilon}^\infty  w^{\rho-1}
e^{- t^{\rho'}\left(\frac{ w^\rho}{\rho}-w\right)}dw$.
It is easy to check that there exists a constant $C>0$ such that,
$A_1\le Ce^{-t^{\rho'}f(1+\epsilon)}$ and 
$A_2\le Ce^{-t^{\rho'}f(1+\epsilon)}$. Furthermore, from the
expansion (\ref{expansion}), we see that,

$$
A_3\le (1+\epsilon)^{\rho-1}e^{c\epsilon^3} e^{\frac{t^{\rho'}}
{\rho'}}\int_{-\infty}^\infty e^{-(\rho-1)t^{\rho'}x^2}dx=
(1+\epsilon)^{\rho-1}e^{c\epsilon^3} e^{\frac{t^{\rho'}}
{\rho'}}\sqrt{\frac{\pi}{(\rho-1)t^{\rho'}}},
$$
for $c=|f^{(3)}(2)|$. Similarly we have that,

$$
A_3\ge 
(1-\epsilon)^{\rho-1}e^{-c\epsilon^3} e^{\frac{t^{\rho'}}
{\rho'}}\left(
\sqrt{\frac{\pi}{(\rho-1)t^{\rho'}}}-O(e^{-t^{\rho'}\epsilon^2})
\right).
$$
Substituting these estimates for $A_1,A_2$ and $A_3$ in (\ref{decomp})
and this in (\ref{expression}), and choosing $\epsilon=t^{-\gamma}$
for $\gamma>0$ small enough, we finish the proof of the lemma.
\end{proof}
\smallskip

Let us finish this section with the following elementary computation.
\begin{lemma}
\label{lemma10}
 Let $k\ge 1$ and $n_1,n_2,\ldots ,n_k$ be natural
numbers larger than $0$. For $u\ge 0$, let
$J_{n_1,\ldots, n_k}(u)=\int_{\sum_{i=1}^k x_k<u}
x_1^{n_1-1}\cdots x_k^{n_k-1} dx_1\cdots dx_k$. Then,
 
$$
\int_0^\infty J_{n_1,\ldots,n_k} (u)e^{-u}du=\prod_{i=1}^k (n_i-1)!.
$$
\end{lemma}

\smallskip

\subsubsection{Path decomposition of annealed first moment}
For $j\ge 1$, call
$\tau_j$ the time of the $j$-th jump of the random walk. Note that
these random times are independent exponential random variables of
rate $2d\kappa$. 
Also,

\begin{equation}
\label{aa}
\left\langle m(0,t)\right\rangle =\sum_{\gamma\in{\mathbb P}(0)}E_0\left[e^{\sum_{x\in{\mathbb Z}^d}
H({\mathcal L}(t,x))}C_\gamma\right],
\end{equation}
where $C_\gamma$ is the event that in the time interval $[0,t]$,
the random walk follows the path $\gamma$. Let us examine a single
term $A_\gamma$, of the summation in (\ref{aa}) corresponding to a
path $\gamma=\{x_1=0,x_2,\ldots, x_n\}$, visiting sites $\{y_1=0,y_2,\ldots,y_k\}$. Without loss of generality, we assume that $y_k=x_n$, so that
$y_k$ is the last visited site.
For each $1\le i\le k$, let us call $n_i$, the number of times
the path visits site $y_i$.
Furthermore, let $i_1,\ldots,i_{n_i}$ be the set of indices
of the set $\{1,2,\ldots,n-1\}$ such that $x_{1_1}=\cdots =x_{i_{n_i}}=y_i$.
Furthermore, note that
$C_\gamma=\bar C_\gamma\cap \{\tau_1+\cdots+\tau_{n-1}< t\}\cap\{\tau_n\ge t-\tau_1-\cdots
\tau_{n-1}\}$, where $\bar C_\gamma=\cap_{j=1}^{n-1}\{X_{\tau_j}=x_{j+1}\}$.  Therefore,
if $S_{n-1}:=\{s_1+\cdots+s_{n-1}<t\}\subset {\mathbb R}^{n-1}$ then,

$$
A_\gamma=\frac{(2d\kappa)^{n-1}}{(2d)^{n-1}}\int_{S_{n-1}}
e^{-2d\kappa( s_1+\cdots + s_{n-1})} 
e^{-2d\kappa(t-s_1-\cdots-s_{n-1})}
e^{H_1+\cdots+H_k}ds_1\cdots ds_{n-1},
$$
where $H_i:=H(s_{i_1}+\cdots +s_{i_{n_i}})$, for $1\le i\le k-1$,
while $H_k:=H(t-v_1-\cdots-v_{k-1})$, with
$v_i=s_{i_1}+\cdots+s_{i_{n_i}}$. Thus,
\begin{eqnarray}
\nonumber
&A_\gamma=\kappa^{n-1} e^{-2d\kappa t}\int_{T_k}\frac{v_1^{n_1-1}}{(n_1-1)!}
\cdots \frac{v_k^{n_k-1}}{(n_k-1)!} e^{H(v_1)+\cdots+H(v_{k-1})+H_k}dv_1\cdots dv_k\\
\label{ae1}
&= \kappa^{n-1}t^n e^{-2d\kappa t}\frac{1}{\prod_{i=1}^k (n_i-1)!}
\int_{T'_k} u_1^{n_1-1}\cdots u_k^{n_k-1}
 e^{H(tu_1)+\cdots+H(tu_{k-1})+ H(t\bar u_k)}du_1\cdots du_k,
\end{eqnarray}
where $T_k:=\{v_1+\cdots+ v_k<t\}$, $T'_k:=\{u_1+\cdots +u_k<1\}$, 
$\bar u_k:=1-u_1-\cdots -u_{k-1}$,
in the second equality we made
the variable change $u_i=t v_i$ and we have used the fact that
$n_1+\cdots +n_k=n$. Let now $0<\delta<1/2$, and define
$W_k=\{u_1+\cdots +u_k<1,\max_{1\le i\le k}u_i<1-\delta\}$
and $V_k=T'_k-W_k$. 
\smallskip
\subsubsection{Asymptotic lower bound}
Here we  compute an asymptotic lower bound for
$A_\gamma$. Note that $ H(t\bar u_k)\ge H(tu_k)$. Hence,

\begin{equation}
\nonumber
A_\gamma\ge \kappa^{n-1}
 t^n e^{-2d\kappa t}\frac{1}{\prod_{i=1}^k (n_i-1)!}
\int_{V_k} u_1^{n_1-1}\cdots u_k^{n_k-1}
 e^{H(tu_1)+\cdots+H(tu_k)}du_1\cdots du_k.
\end{equation}
Now, note that $V_k=\cup_{j=1}^k V_{k,j}$, where the union is disjoint and
$V_{k,j}:=\{u_1+\cdots+u_k<1, u_j\ge 1-\delta  \}$.
Therefore, 
$A_\gamma\ge 
 t^n e^{-2dt}\frac{1}{\prod_{i=1}^k (n_i-1)!} \sum_{j=1}^k I_{k,j}$
where,

\begin{equation}
\nonumber
I_{k,j}:=\int_{V_{k,j}} u_1^{n_1-1}\cdots u_k^{n_k-1}
 e^{H(tu_1)+\cdots+H(tu_k)}du_1\cdots du_k.
\end{equation}
By symmetry it is enough to examine $I_{k,1}$. From lemma 
\ref{cumulant} and (\ref{ae1}) we  obtain,
\begin{equation}
\nonumber
I_{k,1}\ge t^{\frac{\rho' }{2}}\left(\frac{\pi}{\rho-1}\right)^{1/2}
\int_{V_{k,1}} u_1^{n_1-1+\frac{\rho'}{2}}u_2^{n_1-1}
\cdots u_k^{n_k-1}
e^{\frac{t^{\rho'}}{\rho'} u_1^{\rho'}+\varepsilon(tu_1)}du_1\cdots du_k.
\end{equation}
Now, the variable change $u'_1:=1-u_1$ transforms the integral in the above
expression to,

$$
I'_{k,1}:=\int_0^\delta\int_{T_{k,1}} (1-u'_1)^{n_1-1+\frac{\rho'}{2}}
u_2^{n_2-1}
\cdots u_k^{n_k-1}  e^{\frac{t^{\rho'}}{\rho'} (1-u'_1)^{\rho'}+\varepsilon(t(1-u'_1))}du'_1du_2\cdots du_k.
$$
where $T_{k,1}:=\{u_2+\cdots u_k<u'_1\}$.
 Note that,
$(1-u'_1)^{\rho'}=1-\rho' u'_1+\rho'(\rho'-1) \bar u^2/2$, for
some $0\le \bar u\le u'_1$. Therefore, 

\begin{equation}
\label{ae13}
1-\rho' u'_1\le (1-u'_1)^{\rho'}\le 1-\rho' u'_1+
\rho'(\rho'-1) {u'}_1^2/2.
\end{equation}
Then, we get a lower bound,

\begin{eqnarray}
\nonumber
&I'_{k,1}\ge e^{\frac{t^{\rho'}}{\rho'}}
\int_0^\delta J_{\gamma,1}(x)(1-x)^{n_1-1+\frac{\rho'}{2}}
 e^{-t^{\rho'}x+\varepsilon(t(1-x))}dx\\
\label{ae12}
&= \frac{1}{t^{\rho'}}e^{\frac{t^{\rho'}}{\rho'}}
\int_0^{\delta t^{\rho'}}
 J_{\gamma,1}\left(\frac{y}{t^{\rho'}}
\right)\left(1-\frac{y}{t^{\rho'}}\right)^{n_1-1+\frac{\rho'}{2}}
 e^{-y+\varepsilon(t(1-y/t^{\rho'}))}dy.
\end{eqnarray}
where for $1\le j\le k$, we define 
$J_{\gamma,j}(x):=\int_{\sum_{i\ne j}u_i<x}\prod_{i\ne j} u_i^{n_i-1}du_i$. Now,

\begin{equation}
\label{jay}
J_{\gamma,1}\left(\frac{y}{t^{\rho'}}\right)=t^{-\rho'(n-n_1)}J_\gamma(y).
\end{equation}
Therefore, by the dominated convergence theorem, we see that
the right-hand side of (\ref{ae12}) is asymptotically equal to,
$
\frac{1}{t^{\rho'(n+1-n_1)}}e^{\frac{t^{\rho'}}{\rho'}}
\int_0^{\infty}
 J_{\gamma,1}(x )
 e^{-x}dx$. We therefore, conclude that $I_{k,1}$ is asymptotically
lower bounded by,

$$
\left(\frac{\pi}{\rho-1}\right)^{1/2}
\frac{1}{t^{\rho'\left(n+\frac{1}{2}-n_1\right)}}e^{\frac{t^{\rho'}}{\rho'}}
\int_0^{\infty}
 J_{\gamma,1}(x )
 e^{-x}dx.
$$
and that $A_\gamma$ is asymptotically lower bounded by,

$$
\frac{\left(\frac{\pi}{\rho-1}\right)^{1/2}}{\prod_{i=1}^k(n_i-1)!}
t^n e^{-2dt}\sum_{j=1}^k
\frac{1}{t^{\rho'\left(n+\frac{1}{2}-n_j\right)}}e^{\frac{t^{\rho'}}{\rho'}}
K_{\gamma,j},
$$
where $K_{\gamma,j}:=\int_0^\infty J_{\gamma,j}(u)e^{-u}du$. Using
lemma \ref{lemma10}, we obtain the desired asymptotic lower bound.
\smallskip
\subsubsection{Asymptotic upper bound} Here we will obtain an asymptotic
upper bound for $A_\gamma$. First, let us examine the
integral,

\begin{equation}
\nonumber
\bar I_{k}:=\int_{\bar W_k} u_1^{n_1-1}\cdots u_k^{n_k-1}
 e^{H(tu_1)+\cdots+H(tu_{k-1})+ H(t\bar u_k)}du_1\cdots du_k,
\end{equation}
where $\bar W_k:=\{u_1+\cdots +u_k<1,\max_{1\le i\le k-1}u_i<1-\delta,
\bar u_k<1-\delta\}$.
Since on $\bar W_k$ we have $\max_{1\le j\le k-1}u_j<1-\delta$
and $\bar u_k<1-\delta$,
the following inequality is satisfied,
 $\frac{1}{\rho'}t^{\rho'}(u_1^{\rho'}
+\cdots +u_{k-1}^{\rho'}+\bar u_k^{\rho'})\le\frac{1}{\rho'}(1-\delta)^{\rho'-1}
t^{\rho'}$.
Hence, the integral $\bar I_{k}$  is upper bounded by,

\begin{equation}
\nonumber
 (t(1-\delta))^{k\frac{\rho'}{2}}\left(\frac{\pi}{\rho-1}\right)^{k/2}
e^{(1-\delta)^{\rho'-1}
\frac{t^{\rho'}}{\rho'}+
k\varepsilon(t(1-\delta))}
\int_{W_k} u_1^{n_1-1}\cdots u_k^{n_k-1}
du_1\cdots du_k.
\end{equation}
Therefore, $\bar I_k\le c(k)e^{(1-\delta)^{\rho'-1}
\frac{t^{\rho'}}{\rho'}}$, for some constant $c(k)$.
Let us now define $\bar V_k:=T'_k-\bar W_k$ and
note that $\bar V_k=\cup_{j=1}^k\bar V_{k,j}$ where
the union is disjoint and
$\bar V_{k,j}:=\{u_1+\cdots +u_k<1,u_j\ge 1-\delta\}$ for
$1\le j\le k-1$ while $\bar V_{k,k}:=
\{u_1+\cdots +u_k<1,\bar u_k\ge 1-\delta\}$.
Then if,

\begin{equation}
\nonumber
\bar I_{k,j}:=\int_{\bar V_{k,j}} u_1^{n_1-1}\cdots u_k^{n_k-1}
 e^{H(tu_1)+\cdots+H(tu_{k-1})+ H(t\bar u_k)}du_1\cdots du_k,
\end{equation}
it follows that,

\begin{equation}
\label{ae14}
A_\gamma\le 
c(k)e^{(1-\delta)^{\rho'-1}
\frac{t^{\rho'}}{\rho'}}+ t^n e^{-2dt}\frac{1}{\prod_{i=1}^k (n_i-1)!} \sum_{j=1}^k \bar I_{k,j}.
\end{equation}
Then, we need to upper bound the integrals $\bar I_{k,j}$. 
Define $u'_1=1-u_1$. By the subadditivity of
the cumulant generating function, note that
$H(tu_2)+\cdots+H(tu_{k-1}) +H(t\bar u_k)\le H(t(u_2+\cdots+u_{k-1}+\bar 
u_k))= H(t (1-u_1))$.
Hence, $\bar I_{k,1}$ is upper bounded by,

$$
t^{\frac{\rho' }{2}}\left(\frac{\pi}{\rho-1}\right)^{1/2}
\int_{\bar V_{k,1}} u_1^{n_1-1+\frac{\rho'}{2}}u_2^{n_1-1}
\cdots u_k^{n_k-1}
e^{\frac{t^{\rho'}}{\rho'} u_1^{\rho'}+H(t(1-u_1))+\varepsilon(tu_1)}du_1\cdots du_k.
$$
But by the second inequality of display (\ref{ae13}), the
integral in the above expression is upper bounded by,

\begin{eqnarray*}
&e^{\frac{t^{\rho'}}{\rho'}}
\int_0^\delta J_{\gamma,1}(x)(1-x)^{n_1-1+\frac{\rho'}{2}}
 e^{-t^{\rho'}x+t^{\rho'}\rho'(\rho'-1)\frac{x^2}{2}+H(tx)+\varepsilon(t(1-x))}dx\\
& =\frac{1}{t^{\rho'}}e^{\frac{t^{\rho'}}{\rho'}}
\int_0^{\delta t^{\rho'}}
 J_{\gamma,1}\left(\frac{y}{t^{\rho'}}
\right)
 e^{-y+\rho'(\rho'-1)\frac{y^2}{2t^{\rho'}}+H(t^{1-\rho'})+\varepsilon(t(1-y/t^{\rho'}))}dy.
\end{eqnarray*}
By (\ref{jay}) and
 the dominated convergence theorem, this is asymptotically
equivalent as $t\to\infty$ to
$
\frac{1}{t^{\rho'(n+1-n_1)}}e^{\frac{t^{\rho'}}{\rho'}}
K_{\gamma,1}$. This provides an asymptotic upper bound for
$\bar I_{k,1}$. A similar argument gives us the asymptotic
upper bounds 
$
\frac{1}{t^{\rho'(n+1-n_1)}}e^{\frac{t^{\rho'}}{\rho'}}
K_{\gamma,j}$,
for $\bar I_{k,j}$, $2\le j\le k$. Combining these estimates
 with  (\ref{ae14}), and using
lemma \ref{lemma10}, finishes the
proof of the asymptotic upper bound.

\smallskip
\subsubsection{Proof of Proposition \ref{corollary1} }
Define ${\mathbb Q}^n(0,0)$ as the set of paths $\gamma\in{\mathbb P}^n(0,0)$
such that $0$ is visited $n_1$ times with $n_1>n/2$. Note that if $n$
is even this set is empty, whereas for $n>1$ odd
all of these paths start and end at $0$, and they are
of the form $\gamma=\{x_1=0,x_2,x_3,\ldots, x_{n-1},x_n=0\}$, with $x_i=0$ for
$i$ odd, $1\le i\le n$.
Let us express the series 

$$
\sum_{n=1}^\infty 
\frac{1}{t^{(\rho'-1)n}}
\sum_{\gamma\in{\mathbb P}^n(0)} 
\sum_{i=1}^k  \frac{t^{\rho' \left(n_i-\frac{1}{2}\right)}}{ (n_i-1)!},
$$
of the right-hand side of display (\ref{i1.1}) as,
$S_1+S_2+S_3$, where

\begin{eqnarray*}
&S_1:=\sum_{n=1}^\infty 
\frac{1}{t^{(\rho'-1)n}}
\sum_{\gamma\in{\mathbb Q}^n(0)} 
  \frac{t^{\rho' \left(n_1-\frac{1}{2}\right)}}{ (n_1-1)!}\\
&S_2:=\sum_{n=2}^\infty 
\frac{1}{t^{(\rho'-1)n}}
\sum_{\gamma\in{\mathbb Q}^n(0)} 
\sum_{i=2}^k  \frac{t^{\rho' \left(n_i-\frac{1}{2}\right)}}{ (n_i-1)!}\\
&S_3:=\sum_{n=2}^\infty 
\frac{1}{t^{(\rho'-1)n}}
\sum_{\gamma\in{\mathbb R}^n(0)} 
\sum_{i=1}^k  \frac{t^{\rho' \left(n_i-\frac{1}{2}\right)}}{ (n_i-1)!},
\end{eqnarray*}
where ${\mathbb R}^n(0):={\mathbb P}^n(0)-{\mathbb Q}^n(0)$. We will show that
only $S_1$ contributes to the final result.

By our previous remarks, note that the summation in $S_1$ runs only over
odd values of $n=2m+1$, $m\ge 0$. Furthermore, $n_1=m+1$ and
$|{\mathbb Q}^n(0)|=(2d)^m$. Therefore,

$$
S_1=\sum_{m=0}^\infty \frac{1}{t^{(\rho'-1)(2m+1)}}
(2d)^m \frac{t^{\rho' \left(m+\frac{1}{2}\right)}}{m!}
=t^{1-\frac{\rho'}{2}}\sum_{m=0}^\infty \frac{(2d t^{2-\rho'})^m }{m!}=
t^{1-\frac{\rho'}{2}}e^{2d t^{2-\rho'}}.
$$
Let us next examine $S_2$. Note that for $\gamma\in{\mathbb Q}^n(0)$,
we have $n_i=1$ for $2\le i\le k$. Again, only the odd terms in the
series count, and we have,

$$
S_2=t^{1-\frac{\rho'}{2}}\sum_{m=1}^\infty\left(\frac{2d}{t^{2(\rho'-1)}}
\right)^m=t^{1-\frac{\rho'}{2}}\frac{2d}{t^{2(\rho'-1)}-2d}\ll S_1.
$$
Now, let $c<1/\rho'$. Let us write $S_3=S_3'+S_3''$, where

\begin{eqnarray*}
&S'_3:=\sum_{n=2}^\infty 
\frac{1}{t^{(\rho'-1)n}}
\sum_{\gamma\in{\mathbb R}^n(0)} 
\sum_{i=1}^k 1(n_i\ge cn) \frac{t^{\rho' \left(n_i-\frac{1}{2}\right)}}{ (n_i-1)!}\\
&S''_3:=\sum_{n=2}^\infty 
\frac{1}{t^{(\rho'-1)n}}
\sum_{\gamma\in{\mathbb R}^n(0)} 
\sum_{i=1}^k  1(n_i<cn) \frac{t^{\rho' \left(n_i-\frac{1}{2}\right)}}{ (n_i-1)!}.
\end{eqnarray*}
Using the bounds $k\le n$ and $|{\mathbb Q}^n(0)|\le (2d)^n$, note that,

$$
S''_3\le\sum_{n=2}^\infty \frac{1}{t^{(\rho'-1)n}} n(2d)^n 
t^{\rho'\left(cn-\frac{1}{2}\right)}=t^{-\frac{\rho'}{2}}
\sum_{n=2}^\infty n\left(2d t^{2(1+\rho(c-1))}\right)^n\ll S_1,
$$
since $1+\rho(c-1)<0$. To estimate $S'_3$, note that all
the paths in ${\mathbb R}^n(0)$, are such that $n_i\le n/2$, for $1\le i\le k$.
Then,

\begin{equation}
\nonumber
S'_3\le t^{-\frac{\rho'}{2}}\sum_{n=2}^\infty
\frac{(2d)^n}{t^{(\rho'-1)n}}\sum_{m=\lfloor cn\rfloor}^{\lfloor 
n/2\rfloor}
\frac{t^{\rho' m}}{(m-1)!}
\le
t^{-\frac{\rho'}{2}}\sum_{m=0}^\infty
\sum_{n=f(m)}^{\infty}
\frac{(2d)^n}{t^{(\rho'-1)n}}
\frac{t^{\rho' m}}{(m-1)!}.
\end{equation}
Now, performing summation by parts, we see that,

\begin{eqnarray*}
&\sum_{m=0}^\infty\sum_{n=f(m)}^{\infty}
\frac{(2d)^n}{t^{(\rho'-1)n}}
\frac{t^{\rho' m}}{(m-1)!}=
\left(\frac{2d}{t^{\rho'-1}}\right)^2+
\sum_{m=1}^\infty
\left(\frac{2d}{t^{\rho'-1}}\right)^{2m}
\frac{t^{\rho'm}}{(m-1)!}\\
&=(2d)^2 t^{2-2\rho'}+(2d)^2 t^{2-\rho'} e^{2d t^{2-\rho'}}.
\end{eqnarray*}
Hence,

$$
S'_3\le (2d)^2 t^{2-5\rho'/2}+(2d)^2 t^{2-3\rho'/2} e^{2d t^{2-\rho'}}
\ll t^{1-\rho'/2}e^{2dt^{2-\rho'}}. 
$$
\medskip

\section{Asymptotic expansion of the scaling function for $1<\rho<(3+\sqrt{17})/2$} 
\label{seccor}
Here we will prove Corollary \ref{corolary0}, giving an expansion
of the scaling function $h$ defined in (\ref{bheq}) for
$1<\rho<(3+\sqrt{17})/2$. We make the calculations
only up to the value $(3+\sqrt{17})/2$  because for $\rho$ larger
than this number extra terms in the expansion of $h$ have to be
computed, and in order to keep the length of this section limited,
we have decided to stop there. There seems to be no straightforward
interpretation on the appearance of this number. On the other hand,
as it will be shown, transitions in the behavior of $h$ occur
for $\rho=2,3$ and the number $(3+\sqrt{17})/2$. Above this last value,
  transitions should appear
at integer values of $\rho$, and additionally we expect
that for some other non-integer
values of $\rho$.
 To prove part $(i)$ of
Corollary \ref{corolary0}, note that when $1<\rho<2$, we have that
$\rho'>2$. Therefore, $M=1$ [c.f. (\ref{mdef})]. It follows from (\ref{gorder})
of Theorem \ref{theorem3.1}, that $h_1(t)=O\left(\frac{1}{t}\right)$. This together with (\ref{gexpansion})
of Theorem \ref{theorem3.1}, shows that

$$
h(t)=\mathcal A_0\tau^{\rho'-1}
-2d\kappa+O\left(\frac{1}{t^{\rho'-1}}\right),
$$
from where using that $\rho'>2$, part $(i)$ of Corollary \ref{corolary0} follows.
Let us now prove parts $(ii)$ and $(iii)$ of Corollary \ref{corolary0}.
Note that for $\epsilon$ small enough

\begin{eqnarray}
\nonumber
&E_\mu\left[e^{-\frac{1}{\rho} B_1(h_0,v)^\rho}1(\max_{e\in E}v(e)> (1-\epsilon)
\mathcal A_0\tau^{\rho'-1})
\right]\\
\nonumber
&\le
e^{-\frac{\mathcal A_0^\rho\tau^{\rho'}}{1+\frac{1}{2d}}}\mu(\max_{e\in E}v(e)> (1-\epsilon)
\mathcal A_0\tau^{\rho'-1})=
e^{-\frac{\mathcal A_0^\rho\tau^{\rho'}}{1+\frac{1}{2d}}}e^{-\frac{\gamma}{\rho'}(1-\epsilon)^\rho
\tau^{\rho'}}\\
\label{rho33}
&=o\left(e^{-\frac{\gamma}{\rho'}\tau^{\rho'}}\right),
\end{eqnarray}
where in the inequality we have used the fact that
$B_1(h_0,v)\ge  \frac{1}{1+\frac{1}{2d}}\mathcal A_0\tau^{\rho'-1}$
and in the last equality that $\epsilon$ is small enough.

It follows
from (\ref{gexpansion}) that

$$
h(t)=\mathcal A_0\tau^{\rho'-1}
-2d\kappa+h_1(t)+O\left(\frac{1}{t^{3(\rho'-1)}}\right).
$$
Now,

\begin{equation}
\label{ache1}
h_1(t):=\frac{1}{\mathcal A_0^{\rho-1}}\left(\frac{\gamma}{\rho'}\tau^{\rho'-1}(t)
+\frac{1}{\tau(t)}\log E_\mu\left[e^{-\frac{1}{\rho} B_1(h_0,v)^\rho}
\right]\right)
\end{equation}
and for $s\ge 0$,

$$
B_1(s,v)=\frac{s+2d\kappa}{1+
\frac{\kappa}{2d\kappa+s}\sum_{e:|e|_1=1}\frac{\kappa}{2d\kappa+(s-v(e))_+}}.
$$
In analogy with the proof of part $(i)$ of Lemma \ref{lemmanov}, but going to a higher
order Taylor expansion, we see that when $\sup_{e\in E} v_0(e)\le\mathcal A_0\tau^{\rho'-1}(1-\epsilon)$ for some $\epsilon>0$, we have

$$
\frac{1}{\rho}B_1^\rho(h_0(t),v)=\frac{\gamma}{\rho'}\tau^{\rho'}
-\frac{\gamma\rho}{\rho'\mathcal A_0}\kappa^2\sum_{e\in E}
\frac{\tau}{\mathcal A_0
\tau^{\rho'-1}-v_0(e)}+O\left(\frac{1}{\tau^{3\rho'-4}}\right).
$$
Now, note that for all $x>0$ one has that

\begin{eqnarray*}
&\frac{\tau}{\mathcal A_0
\tau^{\rho'-1}-x}=
\mathcal A_0^{-1}\tau^{2-\rho'}+\mathcal A_0^{-2}x\tau^{3-2\rho'}+\mathcal A_0^{-3}x^2\frac{\tau^{4-3\rho'}}{1-\frac{x}{\tau^{\rho'-1}}}\\
&=
\mathcal A_0^{-1}\tau^{2-\rho'}+\mathcal A_0^{-2}x\tau^{3-2\rho'}+
x^2O\left(\tau^{4-3\rho'}\right)
\end{eqnarray*}
It follows that

\begin{eqnarray}
\nonumber
&E_\mu\left[e^{-\frac{1}{\rho} B_1(h_0,v)^\rho}1(\max_{e\in E}v(e)\le (1-\epsilon)
\mathcal A_0\tau^{\rho'-1})
\right]=e^{-\frac{\gamma}{\rho'}\tau^{\rho'}+2d\frac{\gamma\rho}{\rho'}\mathcal A_0^{-2}\kappa^2\tau^{2-\rho'}}\\
\label{rho11}
&\times
\left(\int_0^{(1-\epsilon)\mathcal A_0\tau^{\rho'-1}}
 e^{
 K_1
x\tau^{3-2\rho'}+
x^2O\left(\tau^{4-3\rho'}\right)}e^{-\frac{1}{\rho}x^\rho}
x^{\rho-1}dx\right)^{2d},
\end{eqnarray}
where

$$
K_1:=\frac{\gamma\rho}{\rho'}\kappa^2\mathcal A_0^{-3}
$$
Now, for $2\le\rho<3$ we claim that

\begin{equation}
\label{rho3}
A:=\int_0^{(1-\epsilon)\mathcal A_0\tau^{\rho'-1}} e^{
K_1 x\tau^{3-2\rho'}+
x^2O\left(\tau^{4-3\rho'}\right)}e^{-\frac{1}{\rho}x^\rho}
x^{\rho-1}dx=1+o(1).
\end{equation}
To prove (\ref{rho3}),  for $\delta<2\rho'-3$ and $t$ large enough, write

\begin{equation}
\label{rhodec}
A=\int_0^{\tau^\delta} g(x)dx
+
\int_{(1-\epsilon)\mathcal A_0\tau^{\rho'-1}}^{\tau^\delta}g(x)dx
\end{equation}
where

$$
g(x):= e^{
K_1x\tau^{3-2\rho'}+
x^2O\left(\tau^{4-3\rho'}\right)}e^{-\frac{1}{\rho}x^\rho}
x^{\rho-1}dx.
$$
Note that for $0\le x\le \tau^\delta$, since
$x\tau^{3-2\rho'}=o(1)$ and also $x^2\tau^{4-3\rho'}=o(1)$,
with $\lim_{t\to\infty}\sup_{0\le x\le \tau^\delta}o(1)=0$,
 we have that

$$
g(x)=e^{o(1)}e^{-\frac{1}{\rho} x^\rho}x^{\rho-1}=
e^{-\frac{1}{\rho} x^\rho}x^{\rho-1}+o(1).
$$
Therefore the first integral in  (\ref{rhodec}) satisfies

$$
\int_0^{\tau^\delta} g(x)dx=1+o(1).
$$
For the second integral in (\ref{rhodec}), remark that

$$
\sup_{\tau^\delta\le x\le (1-\epsilon)\mathcal A_0\tau^{\rho'-1}}
x^2O\left(\tau^{4-3\rho'}\right)=o(1).
$$
On the other hand, the function

\begin{equation}
\nonumber
u(x):=K_1x\tau^{3-2\rho'}
-\frac{1}{\rho}x^\rho,
\end{equation}
is decreasing in the interval $[\tau^\delta,(1-\epsilon)\mathcal A_0\tau^{\rho'-1}]$, so that

$$
\int_{(1-\epsilon)\mathcal A_0\tau^{\rho'-1}}^{\tau^\delta}g(x)dx=O\left(
e^{-C\tau^{\delta\rho}}\right)=o(1),
$$
for some constant $C>0$. This finishes the proof
of (\ref{rho3}). Substituting now (\ref{rho3}) into (\ref{rho11})
we conclude that

\begin{equation}
\label{rho22}
E_\mu\left[e^{-\frac{1}{\rho} B_1(h_0,v)^\rho}1(\max_{e\in E}v(e)\le (1-\epsilon)
\mathcal A_0\tau^{\rho'-1})
\right]=e^{-\frac{\gamma}{\rho'}\tau^{\rho'}+2d\frac{\gamma\rho}{\rho'}\mathcal A_0^{-2}\kappa^2\tau^{2-\rho'}}(1+o(1)).
\end{equation}
Combining (\ref{rho22}) with (\ref{rho33}) we conclude that

\begin{equation}
\label{rhol}
E_\mu\left[e^{-\frac{1}{\rho} B_1(h_0,v)^\rho}
\right]=e^{-\frac{\gamma}{\rho'}\tau^{\rho'}+2d\frac{\gamma\rho}{\rho'}\mathcal A_0^{-2}\kappa^2\tau^{2-\rho'}}(1+o(1)).
\end{equation}
Substituting (\ref{rhol}) back into (\ref{ache1}) we see that
for $2\le\rho< 3$,

$$
h_1(t)=
2d
 \kappa^2\mathcal A_0^3 \tau^{1-\rho'}(t)+o\left(\frac{1}{t}\right),
$$
which proves part $(ii)$.

To prove part $(iii)$, consider
the function

\begin{equation}
\label{erf}
r(x):=\mathcal A_0^2 K_1\frac{\tau}{\mathcal A_0\tau^{\rho'-1}-x}-\frac{1}{\rho}x^\rho+(\rho-1)\log x,
\end{equation}
definde for $x>0$.
We will establish
the following lemma.

\medskip
\begin{lemma}
\label{rexpan} Let $3<\rho\le 4$.
Then, the following are satisfied.

\begin{itemize}

\item[(i)] The function $r$ defined in (\ref{erf}) has
a global maximum $x_t$  on the interval $[0,(1-\epsilon)\mathcal A_0\tau^{\rho'-1}]$
 where its derivative vanishes and
such that

\begin{equation}
\label{ext}
x_t=\mathcal (A_0 K_1)^{\frac{1}{\rho-1}} \tau^{\frac{1}{\rho-1}(3-2\rho')}-
\frac{1}{\mathcal A_0 K_1}\tau^{-(3-2\rho')}+o\left(\tau^{-(3-2\rho')}\right).
\end{equation}

\item[(ii)] The function $r$ satisfies

\begin{eqnarray*}
&r(x_t)=\mathcal A_0 K_1\tau^{2-\rho'}+K_1(\mathcal A_0 K_1)^{\frac{1}{\rho-1}}
\left(1-\frac{1}{\rho}\mathcal A_0\right)\tau^{\frac{\rho}{\rho-1}(3-2\rho')}\\
&+(3-2\rho')\log\mathcal A_o\tau-\mathcal A_0^2 K_1^2+o(1).
\end{eqnarray*}

\item[(iii)] For every $t>0$ we have that

\begin{equation}
\label{conclr}
E_\mu\left[e^{-\frac{1}{\rho} B_1(h_0,v)^\rho}
\right]=e^{r(x_t)}\sqrt{\frac{\pi (\rho-1)}{\tau^{\frac{\rho-2}{\rho-1}(3-2\rho')}}}
\left(1+o(1)\right).
\end{equation}

\end{itemize}
\end{lemma}

\medskip

\begin{proof}

\noindent {\it Proof of parts $(i)$ and $(ii)$}. Note that

\begin{eqnarray*}
&r'(x)=\mathcal A_0^2K_1\frac{\tau}{(\mathcal A_0\tau^{\rho'-1}-x)^2}-x^{\rho-1}+
(\rho-1)\frac{1}{x}\quad{\rm and}\\
&
r''(x)=2\mathcal A_0^2K_1\frac{\tau}{(\mathcal A_0\tau^{\rho'-1}-x)^3}
-(\rho-1)x^{\rho-2}-(\rho-1)\frac{1}{x^2}.
\end{eqnarray*}
Now, for every $\epsilon'>0$ we have that
$r'(x)>0$ whenever $x\le\bar x:=\tau^{\frac{1-\epsilon'}{\rho-1}(3-2\rho')}$, while
$r'(x)<0$ whenever $x\ge (1-\epsilon)\mathcal A_0\tau^{\rho'-1}$.
On the other hand, it is easy to check that
$r''(x)<0$ for $\bar x\le x\le (1-\epsilon)\mathcal A_0\tau^{\rho'-1}$.
It follows that there exists only one
 root $x_t$ of the equation $r'(x)=0$ on the interval
$[0,(1-\epsilon)\mathcal A_0\tau^{\rho'-1})$.
To prove (\ref{ext}), as a first step, we note that

$$
x_t=(\mathcal A_0K_1)^{\frac{1}{\rho-1}}\tau^{\frac{1}{\rho-1}(3-2\rho')}
+y_t,
$$
where $y_t=o\left(\tau^{\frac{1}{\rho-1}(3-2\rho')}\right)$. Furthermore

\begin{equation}
\label{2ka1}
2 K_1\tau^{4-3\rho'}x_t-(\rho-1)x_t^{\rho-2}y_t-(\rho-1)\frac{1}{x_t}=u(t),
\end{equation}
where $u(t)$ is of smaller order in $t$ than the three terms of the
left-hand side of (\ref{2ka1}). Now, for $\rho<\frac{3+\sqrt{17}}2$,
the last term of the left-hand side of (\ref{2ka1}) has a higher order
than the first term. This implies that 

$$
y_t=-\frac{1}{x_t^{\rho-1}}+o\left(\frac{1}{x_t^{\rho-1}}\right)
=-\frac{1}{\mathcal A_0 K_1\tau^{(3-2\rho')}}+
o\left(\frac{1}{\tau^{(3-2\rho')}}\right),
$$
which proves (\ref{ext}) of part $(i)$. The proof of part $(ii)$
now follows using the expansion (\ref{ext}) of $x_t$ of part $(i)$.

\noindent {\it Proof of part $(iii)$}. By a standard Taylor expansion, we see that
for every real $y$ such that $x_t-|y|>0$, there is a $\vartheta\in [x_t-|y|,x_t+|y|]$
such that
$$
r(x_t+y)=r(x_t)+
\frac{y^2}{2}r''(x_t)+
\frac{y^3}{6}r'''(\vartheta).
$$
Note that

\begin{equation}
\nonumber
r'''(x)=6\mathcal A_0^2K_1\frac{\tau}{(\mathcal A_0\tau^{\rho'-1}-x)^4}
-(\rho-2)(\rho-1)x^{\rho-3}+2(\rho-1)\frac{1}{x^3}.
\end{equation}
Therefore,

$$
r''(x_t)=-(\rho-1)x_t^{\rho-2}+O\left(t^{4-3\rho'}\right)
$$
and for $|\vartheta| \le 2x_t$,

\begin{equation}
\label{thirdorder}
r'''(\vartheta)=-(\rho-1)(\rho-2)\vartheta^{\rho-3}+O\left(t^{5-4\rho'}\right).
\end{equation}
It follows that

\begin{eqnarray}
\nonumber
&
E_\mu\left[e^{-\frac{1}{\rho} B_1(h_0,v)^\rho}
1(\max_{e\in E}v(e)\le (1-\epsilon)
\mathcal A_0\tau^{\rho'-1})
\right]=\int_0^{(1-\epsilon)\mathcal A_o\tau^{\rho'-1}}e^{r(x)}dx\\
\nonumber
&
=
e^{r(x_t)}\int_{-x_t}^{(1-\epsilon)\mathcal A_o\tau^{\rho'-1}-x_t}e^{
\frac{1}{2}y^2r''(x_t)+\frac{1}{6}y^3r'''(\vartheta)}dy\\
\label{expe11}
&
=
e^{r(x_t)}\int_{-x_t}^{(1-\epsilon)\mathcal A_o\tau^{\rho'-1}-x_t}e^{
-\frac{\rho-1}{2}y^2x^{\rho-2}_t++y^2O\left(t^{4-3\rho'}\right)+\frac{1}{6}y^3r'''(\vartheta)}dy.
\end{eqnarray}
For $\delta$ such that $\tau^\delta\le x_t$ write

\begin{eqnarray}
\nonumber
&
\int_{-x_t}^{(1-\epsilon)\mathcal A_o\tau^{\rho'-1}-x_t}e^{
-\frac{\rho-1}{2}y^2x^{\rho-2}_t++y^2O\left(t^{4-3\rho'}\right)+\frac{1}{6}y^3r'''(\vartheta)}dy\\
\label{intdec}
&\!\!\!\! =
\int_{-\tau^\delta}^{\tau^\delta}e^{
-\frac{\rho-1}{2}y^2x^{\rho-2}_t++y^2O\left(t^{4-3\rho'}\right)+\frac{1}{6}y^3r'''(\vartheta)}dy
+
\int_{B_\delta}e^{
-\frac{\rho-1}{2}y^2x^{\rho-2}_t++y^2O\left(t^{4-3\rho'}\right)+\frac{1}{6}y^3r'''(\vartheta)}dy,
\end{eqnarray}
where $B_\delta:=\{y:|y|\ge\tau^\delta,-x_t\le y\le
(1-\epsilon)\mathcal A_o\tau^{\rho'-1}-x_t\}$.
For the first integral in the right-hand side of (\ref{intdec}), we
have that

\begin{eqnarray}
\nonumber
&\int_{-\tau^\delta}^{\tau^\delta}e^{
-\frac{\rho-1}{2}y^2x^{\rho-2}_t++y^2O\left(t^{4-3\rho'}\right)+\frac{1}{6}y^3r'''(\vartheta)}dy\\
&=\frac{1}{\sqrt{x_t^{\rho-2}}}
\int_{-\tau^\delta \sqrt{x_t^{\rho-2}}}^{\tau^\delta \sqrt{x_t^{\rho-2}}}e^{
-\frac{\rho-1}{2}y^2+y^2O\left(\frac{t^{4-3\rho'}}{x_t^{\rho-2}}\right)+\frac{1}{6 x_t^{3(\rho-2)/2}}y^3r'''(\vartheta)}dy\\
\label{intdec11}
&=\left(1+o(1)\right)
\frac{1}{\sqrt{x_t^{\rho-2}}}
\int_{-\tau^\delta \sqrt{x_t^{\rho-2}}}^{\tau^\delta \sqrt{x_t^{\rho-2}}}e^{
-\frac{\rho-1}{2}y^2+\frac{1}{6 x_t^{3(\rho-2)/2}}y^3r'''(\vartheta)}dy.
\end{eqnarray}
where we have used the fact that when $\rho>3$, one has that
$y^2O\left(t^{4-3\rho'}/x^{\rho-2}_t\right)\le Ct^{2\delta+4-3\rho'}=o(1)$.
Now

\begin{eqnarray}
\nonumber
&\frac{1}{\sqrt{x_t^{\rho-2}}}
\int_{-\tau^\delta \sqrt{x_t^{\rho-2}}}^{\tau^\delta \sqrt{x_t^{\rho-2}}}e^{
-\frac{\rho-1}{2}y^2+\frac{1}{6 x_t^{3(\rho-2)/2}}y^3r'''(\vartheta)}dy\\
\nonumber
&=\frac{1}{\sqrt{x_t^{\rho-2}}}
\int_{-\tau^\delta}^{\tau^\delta}e^{
-\frac{\rho-1}{2}y^2+\frac{1}{6 x_t^{3(\rho-2)/2}}y^3r'''(\vartheta)}dy\\
\label{intdec2}
&+
\frac{1}{\sqrt{x_t^{\rho-2}}}
\int_{D_\delta}e^{
-\frac{\rho-1}{2}y^2+\frac{1}{6 x_t^{3(\rho-2)/2}}y^3r'''(\vartheta)}dy,
\end{eqnarray}
where $D_\delta:=\{y:|y|\ge\tau^\delta,-\tau^\delta\sqrt{x_t^{\rho-2}}\le y\le
\tau^\delta \sqrt{x_t^{\rho-2}}\}$. For the first integral
of the right-hand side of (\ref{intdec2}), we have that

\begin{equation}
\label{eqqq1}
\frac{1}{\sqrt{x_t^{\rho-2}}}
\int_{-\tau^\delta}^{\tau^\delta}e^{
-\frac{\rho-1}{2}y^2+\frac{1}{6 x_t^{3(\rho-2)/2}}y^3r'''(\vartheta)}dy
=\sqrt{\frac{\pi (\rho-1)}{x_t^{\rho-2}}}+o\left(x_t^{-(\rho-2)/2}\right),
\end{equation}
where we have used the fact that by (\ref{thirdorder})
we have that $|y^3|x_t^{-3(\rho-2)/2}r'''(\vartheta)\le
C\tau^{3\delta}x_t^{\rho-3-3(\rho-2)/2}\le C t^{3\delta-1/2}=o(1)$ for
$\delta<1/6$. For the second integral on the right-hand
side of (\ref{intdec2}), note that since
 $\frac{1}{x_t^{3(\rho-2)/2}}y^3r'''(\vartheta)
\le  y^2 o(t)$ uniformly for $y\in D_{\delta}$, we have

\begin{equation}
\label{eqqq2}
\frac{1}{\sqrt{x_t^{\rho-2}}}
\int_{D_\delta}e^{-\frac{\rho-1}{2}y^2+\frac{1}{6 x_t^{3(\rho-2)/2}}y^3r'''(\vartheta)}dy=
\frac{1}{\sqrt{x_t^{\rho-2}}}O\left(e^{-\tau^{2\delta}}\right)=o\left(x_t^{-(\rho-2)/2}\right).
\end{equation}
Substituting (\ref{eqqq1}) and (\ref{eqqq2}) into (\ref{intdec2}), (\ref{intdec11}), (\ref{intdec}) and (\ref{expe11}), and using (\ref{rho33}) together
with (\ref{ext}) we conclude that (\ref{conclr}) of part $(ii)$ of Lemma \ref{rexpan} is satisfied.

\end{proof}

\medskip
\appendix
\section{Abstract rank-one perturbation theory}
\label{fr}

 For the sake of completeness, we review
here the standard rank-one perturbation theory (see
\cite{kato,simon} for an overview). Let ${  H}_0$ be a bounded
self-adjoint operator in a Hilbert space
${\mathcal H}$. Here we want to establish some cases under which
a rank-one self-adjoint perturbation $H$ of $H_0$, has
a principal eigenvalue and eigenfunction possibly with a
series expansion on some small parameter.
We will need  the resolvents ${  R}_\lambda:=(\lambda I-{  H})^{-1}$ and
${  R}^0_\lambda:=(\lambda I-{  H}_0)^{-1}$, of ${  H}$
and  ${  H}_0$ respectively,
 defined for $\lambda$ not in the corresponding spectrums
$\sigma ({  H})$ and $\sigma ({  H}_0)$. Let us denote by $res({  H}_0)$ and
$res({  H})$ the respective resolvent sets.
The top of the
spectrum of ${  H}_0$, will be denoted by

$$
\lambda^0_+:=\sup\{\lambda:
\lambda\in \sigma ({  H}_0)\}.
$$

\smallskip

\subsection{Definitions} 
\label{rankone}
Let us consider a rank one perturbation of ${  H}_0$ depending on
a {\it large} parameter $h>0$:
 
\begin{equation}
\nonumber
{  H}:={  H}_0+h{  B},\qquad {  B}:=(\phi,\cdot )\phi
\end{equation}
for some normalized $\phi\in{\mathcal H}$. Note that ${  H}$
 is also bounded and self-adjoint.
We will
 show that if $h$ is large enough,  ${  H}$
has a principal eigenvalue and eigenfunction with a Laurent series
expansion on $h$. 
Define then the following
two families of elements of ${\mathcal H}$,

$$
 r_\lambda :=(\lambda I-{  H})^{-1}\phi, \qquad \lambda\notin \sigma ({  H}),
$$

$$
q_\lambda :=(\lambda I-{  H}_0)^{-1}\phi, \qquad \lambda\notin \sigma
 ({  H}_0).
$$

\subsection{The Aronszajn-Krein formula}
Here we will state and prove the famous Aronzajn-Krein formula (see
for example \cite{simon}), in our particular context.
Let us first define the following set,

$$
{\mathcal S}:=\{\lambda\in res({  H}_0): h(\phi,q_\lambda)=1\}.
$$
and the quantity,

\begin{equation}
\label{key}
h_0:=\frac{1}{\lim_{\lambda\searrow\lambda^0_+}(\phi,  q_\lambda )}.
\end{equation}
 Note that  $(\phi,q_\lambda)$ is
decreasing in $\lambda$ for $\lambda>\lambda^0_+$.
Indeed,
$\frac{d(\phi,q_\lambda)}{d\lambda}=-||q_\lambda||^2<0$, since
$\phi\ne 0$. Hence, the
limit in display (\ref{key}) exists,
possibly having the value $\infty$.
In the sequel, we will interpret the quantity $h_0$ as $0$
when the limit in the denominator of the right hand side
of (\ref{key}) is $\infty$. 

\begin{lemma} 
\label{es}
${\mathcal S}\subset{\mathbb R}$, and has only isolated points. Furthermore,
there is  a $\lambda\in {\mathcal S}$ such that $\lambda>\lambda^0_+$ if and only if 
$h>h_0$. In this case it is unique.
\end{lemma}
\begin{proof} Note that $h(\phi,q_\lambda)-1$ is an analytic function on the open
set $res({  H}_0)$. Therefore, its zeros are isolated. On the other hand, since $H$
is self-adjoint, they have to be real. The last statement follows from the
fact that $(\phi,q_\lambda)$ is decreasing if $\lambda>\lambda^0_+$ and
$\lim_{\lambda\to\infty}(\phi,q_\lambda)=0$.
\end{proof}
\smallskip

\begin{theorem} 
\label{akform}
 Consider the bounded
selfadjoint operators ${  H}_0$ and ${  H}$. 

\begin{itemize}

\item[(i)] {\bf  Aronszajn-Krein formula.}
 If $\lambda\notin \sigma ({  H}_0)\cup\sigma ({  H})$ then,

\begin{equation}
\label{ak2}
{  R}_\lambda={  R}^0_\lambda+\frac{h}{1-(\phi,q_\lambda)h}(q_\lambda,\cdot)q_\lambda.
\end{equation}

\item[(ii)] {\bf Spectrum of ${  H}$.} \begin{equation}
\label{ak3}
{\mathcal S}\subset \sigma({  H})\subset {\mathcal S}\cup \sigma ({  H}_0).
\end{equation}

\end{itemize}
\end{theorem}

\begin{proof}  Let us first prove part $(i)$ and the first inclusion of part $(ii)$.
Assume that $\lambda\in  res({  H}_0)\cap res({  H})$. 
By definition we have, $(\lambda I-{  H}) r_\lambda =\phi$.
Hence,

\begin{equation}
\nonumber
(\lambda I-{  H}_0) r_\lambda =
(1+(\phi, r_\lambda)h)\phi.
\end{equation}
Making the resolvent $(\lambda I-{  H}_0)^{-1}$ act on both sides of
this equality, we get,

\begin{equation}
\label{ro1}
r_\lambda =(1+(\phi, r_\lambda)h)  q_\lambda.
\end{equation}
Taking the scalar product with $\phi$,
we see that $
(1- (\phi,q_\lambda)h)(\phi,r_\lambda)= (\phi, q_\lambda)$.
This shows that $(\phi, q_\lambda)h\ne 1$ and hence,

\begin{equation}
\label{s1}
(\phi, r_\lambda )=\frac{ (\phi, q_\lambda )}{1-(\phi,
 q_\lambda)h}.
\end{equation}
Therefore, ${\mathcal S}\subset \sigma ({  H})$. 
Substituting (\ref{s1})  back
in the identity (\ref{ro1}) and using
${  R}_\lambda={  R}_\lambda^0+hr_\lambda(q_\lambda,\cdot)$
proves (\ref{ak2}).
Now, assume that $\lambda\notin {\mathcal S}\cap\sigma ({  H}_0)$. Then the right hand
side of (\ref{ak2}) is well defined as a bounded selfadjoint operator
in ${\mathcal H}$. A simple computation shows that it is the inverse
of the operator $(\lambda I-{  H})$.
\end{proof}

\smallskip

\noindent 
 From theorem \ref{akform} we can now deduce the following
corollary.
\smallskip

\begin{corollary} 
\label{t1}
Either of the following is true:

\begin{itemize}

\item[(i)] If $h>h_0$,
 ${  H}$ has a unique simple eigenvalue $\lambda_{max}>\lambda^0_+$ and
$\sigma({  H})/\{\lambda_{max}\}\subset (-\infty,\lambda^0_+]$.

\item[(ii)] If $h\le h_0$, then $\sigma ({  H})\subset (-\infty,\lambda^0_+]$.

\end{itemize}

\noindent Furthermore, if $(i)$ is satisfied the eigenfunction
of $\lambda_{max}$ is proportional to
$q_{\lambda_{max}}$ and
there exist an $r_0>h_0$ such that
$\lambda_{max}$ admits a Laurent series expansion for $h>r_0$,

$$
\lambda_{max}=h +\sum_{k=0}^\infty \frac{b_k}{h^k}.
$$

\end{corollary}

\begin{proof} If $h\le h_0$, by lemma \ref{es} the equation
$h(\phi,q_{\lambda})=1$ does not have any solution
$\lambda>\lambda^0_+$. By theorem \ref{akform}, there is no
$\lambda\in\sigma({  H})$ such that $\lambda>\lambda^0_+$.
On the other hand, by lemma \ref{es}, if $h>h_0$,
there is a unique $\lambda_{max}>\lambda^0_+$
 such that $h(\phi,q_{\lambda_{max}})=1$. 
By theorem \ref{akform}, $\lambda_{max}\in\sigma (H_0)$ and
the spectral
projector of ${  H}$ on $\lambda_{max}$ is given by,

$$
{  P}=\frac{1}{||q_{\lambda_{max}}||^2}(q_{\lambda_{max}},\cdot)q_{\lambda_{max}}.
$$
This shows that the
eigenfunction of $\lambda_{max}$ is proportional
to $q_\lambda$. Finally, defining $u:=1/h$, we see that
if $\lambda(u)$ satisfies $(\phi, q_{\lambda(u)})=u$, then

$$
\frac{d(\phi, q_{\lambda(u)})}{du}=1.
$$
By the implicit function theorem, this implies that
there is a neighborhood of the point $u=0$, where the function
$1/\lambda(u)$ is analytic.  
\end{proof}

\end{document}